\documentclass[12pt,a4paper]{amsart}

\usepackage{amsmath,amsthm,amssymb}
\usepackage{enumerate}

\usepackage[top=37mm,bottom=37mm,left=27mm,right=27mm]{geometry}

\newtheorem{theorem}{Theorem}[section]
\newtheorem{lemma}[theorem]{Lemma}
\newtheorem{proposition}[theorem]{Proposition}
\newtheorem{corollary}[theorem]{Corollary}
\newtheorem{thmx}{Theorem}

\theoremstyle{remark}
\newtheorem{remark}[theorem]{Remark}

\newtheorem*{claim*}{Claim}

\newcommand{\Exp}{\operatorname{Exp}}
\newcommand{\II}{\mathbb{II}}
\newcommand{\id}{\operatorname{id}}
\newcommand{\g}{\mathfrak}
\newcommand{\ad}{\operatorname{ad}}
\newcommand{\Ad}{\operatorname{Ad}}

\newcommand{\Span}{\operatorname{span}}
\newcommand{\R}{\mathbb{R}}

\begin{document}
\title[Codimension two polar homogeneous foliations]
{Codimension two polar homogeneous foliations on symmetric spaces of noncompact type}

\author[J.~C.\ D\'{\i}az-Ramos]{Jos\'e Carlos D\'{\i}az-Ramos}
\author[J.~M.\ Lorenzo-Naveiro]{Juan Manuel Lorenzo-Naveiro}

\address{Department of Mathematics, CITMAGA, University of Santiago de Compostela, Spain}

\email{josecarlos.diaz@usc.es}
\email{jm.lorenzo@usc.es}

\begin{abstract}
We classify homogeneous polar foliations of codimension two on irreducible symmetric spaces of noncompact type up to orbit equivalence.
Any such foliation is either hyperpolar or the canonical extension of a polar homogeneous foliation on a rank one symmetric space.
\end{abstract}




\thanks{Supported by projects PID2022-138988NB-I00/AEI/10.13039/501100011033 (Spain) and ED431F 2020/04, ED431C 2023/31 (Xunta de Galicia, Spain). The second author
acknowledges support of the FPI program (Spain).}

\maketitle
\section{Introduction}

A proper isometric action of a Lie group $G$ on a Riemannian manifold $M$ is said to be \emph{polar} if there exists a connected and complete submanifold $\Sigma$ that intersects every orbit orthogonally. The action is \emph{hyperpolar} if it admits a flat section (with the induced metric). Many results in Algebra and Geometry can be thought of in terms of polar actions. For instance, the system of polar coordinates in the Euclidean plane $\R^{2}$ is connected to the polarity of the standard representation of the orthogonal group $\mathsf{O}(2)$ in the plane. In a similar fashion, the spectral theorem for hermitian matrices can be restated as follows: the action of the unitary group $\mathsf{U}(n)$ on the vector space of $n \times n$ hermitian matrices by conjugation is hyperpolar; a section is given by the subspace of diagonal matrices. 
Based on this last example, we usually refer to the points of a section as canonical forms of the elements of the ambient space (see~\cite{PalaisTerngForms} for more details). A final example with a more Lie-theoretic flavor can be seen in the adjoint action of a compact Lie group, whose sections are the maximal tori.

Broadly speaking, we should not expect a generic Riemannian manifold to admit a nontrivial polar action. Indeed, sections are known to be totally geodesic submanifolds of the ambient space, so the existence of polar actions implies the existence of (sufficiently large) families of globally defined Killing fields and totally geodesic submanifolds orthogonal to each other; both objects are rare in a space with no restrictions~\cite[Theorem~A]{MurphyWilhelm}. Thus, it should not come as a surprise that polar actions can only be possible in Riemannian spaces with a certain degree of symmetry. Some structural results concerning spaces that do admit polar actions can be seen for example in~\cite[Theorem~A]{FangGroveThorbergsson}, and~\cite[Theorem~1]{KollrossPodesta}, and its subsequent generalization in~\cite[Theorem~1.1]{DiScala}. 
Consequently, the most natural classes of Riemannian manifolds where we can develop a fruitful study of these actions are those of homogeneous and symmetric spaces.

Our main interest in this paper concerns the classification of polar actions (up to orbit equivalence) on manifolds with symmetries.
Dadok~\cite{Dadok} proved that any polar representation on a Euclidean space is orbit equivalent to the isotropy representation of a symmetric space, thus solving the problem in the round sphere $\mathsf{S}^{n}$.
The classification in Euclidean spaces follows easily from Dadok's result.
The corresponding classification of polar actions in the real hyperbolic space $\R\mathsf{H}^{n}$ was given by Wu in~\cite{Wu}.

As for symmetric spaces, while the problem is nearing its full conclusion in the compact setting (see~\cite{Kollross-Lytchak} and the references therein), little is known about polar actions on noncompact symmetric spaces, our main interest for this article.
Nevertheless, there are already complete classifications in real and complex hyperbolic spaces~\cite{DiazDominguezKollross20}. 
For the special case of cohomogeneity one actions, we have the works of Berndt and Tamaru (see~\cite{BerndtTamaruC1} and the references in it) as well as a more recent structural result given by the first author, Dom\'inguez-V\'azquez and Otero~\cite{DiazDominguezOtero22}.

We say that the orbits of a proper action form a \emph{homogeneous foliation} if they all have the same orbit type. Cohomogeneity one homogeneous foliations have been classified up to orbit equivalence in irreducible noncompact symmetric spaces by Berndt and Tamaru in~\cite{BerndtTamaruC1Foliations}. 
This result has been extended by Solonenko~\cite{Solonenko} for the reducible case. 
In addition, Berndt, Tamaru and the first author give in~\cite{BerndtDiazTamaru10} a procedure to construct all possible hyperpolar homogeneous foliations on any symmetric space of noncompact type;
it is also shown that there are polar foliations on noncompact irreducible symmetric spaces of higher rank that are not hyperpolar.
The results presented in these two papers make extensive use of tools from the theory of real semisimple Lie algebras, especially the root space and Iwasawa decompositions for these objects. At any rate, we are yet to obtain general results concerning polar homogeneous foliations that are not hyperpolar.

The aim of this paper is to start the study of polar nonhyperpolar homogeneous foliations on symmetric spaces of noncompact type. 
We determine all of these that have the hyperbolic plane $\R \mathsf{H}^{2}$ as a section. As a consequence, combining this work with~\cite{BerndtDiazTamaru10}, we obtain a list of all polar homogeneous foliations of codimension two in any such space. In order to state the main result of this paper we need to introduce some concepts.
See Section~\ref{Sec:prelims} for further details.

A Riemannian symmetric space of noncompact type can be written as a quotient $M=G/K$, where $G$ is a semisimple Lie group and $K$ is the isotropy group at a point $o\in M$.
The Lie algebra $\g{g}$ of $G$ has a Cartan decomposition $\g{g}=\g{k}\oplus\g{p}$, where $\g{p}\cong T_{o}M$ is the orthogonal complement of $\g{k}$ in $\g{g}$ with respect to the Killing form. We normalize the metric on $M$ so that its restriction to $\g{p}$ coincides with the Killing form of $\g{g}$ restricted to $\g{p}$.
We choose a maximal abelian subspace $\g{a}$ of $\g{p}$.
This determines a root space decomposition $\g{g}=\g{g}_0\oplus\bigl(\bigoplus_{\lambda\in\Sigma}\g{g}_\lambda\bigr)$, where $\Sigma$ is the set of roots with respect to $\g{a}$.
We choose a positivity criterion on $\Sigma$ and denote by $\Sigma^+$ the set of positive roots.
Let $\Lambda\subseteq\Sigma^+$ be the corresponding set of simple roots.
We define $\g{n}=\bigoplus_{\lambda\in\Sigma^+}\g{g}_\lambda$.
Then we have an Iwasawa decomposition $\g{g}=\g{k}\oplus\g{a}\oplus\g{n}$. 

Let $\Phi\subseteq\Lambda$ be a subset of simple roots.
Then $\Phi$ determines a so-called parabolic subgroup $Q_\Phi$ with Langlands decomposition $Q_\Phi=M_\Phi A_\Phi N_\Phi$, where $M_\Phi$ is reductive, $A_\Phi$ is abelian, and $N_\Phi$ is nilpotent.
The totally geodesic submanifold $B_\Phi=M_\Phi\cdot o$ is a symmetric space of noncompact type of rank $\lvert\Phi\rvert$ that is called the boundary component associated with $\Phi$.
If $H_\Phi$ is a subgroup of the isometry group of $B_\Phi$, then $H_\Phi A_\Phi N_\Phi$ induces an isometric action on $M$.
We call this action the canonical extension of the isometric action on the boundary component $B_\Phi$ to the symmetric space $M$.  See~\cite{BerndtTamaruC1} for further details.

The main result of this paper is the following:

\begin{thmx}\label{th:MainTheoremCanonicalExtension}
	Let $M$ be a connected irreducible Riemannian symmetric space of noncompact type.
	Every codimension two polar nonhyperpolar homogeneous foliation on $M$ is orbit equivalent to the canonical extension of a codimension two polar nonhyperpolar homogeneous foliation on a boundary component of rank one in $M$.
\end{thmx}

More explicitly, we prove:

\begin{thmx}\label{th:MainTheoremList}
Let $M=G/K$ be a connected irreducible Riemannian symmetric space of noncompact type.
Then, a codimension two homogeneous polar nonhyperpolar foliation on $M$ is orbit equivalent to the orbit foliation of the closed connected Lie group whose Lie algebra is given by one of the following possibilities:
\begin{enumerate}[\rm (i)]
\item $(\ker\alpha)\oplus(\g{n}\ominus\ell_\alpha)$, where $\alpha\in\Lambda$ is a simple root, and $\ell_\alpha$ is a line in $\g{g}_\alpha$, or\label{th:main:ker+n-l}
\item $\g{a}\oplus(\g{n}\ominus\g{v}_\alpha)$, where $\alpha\in\Lambda$ is a simple root, and $\g{v}_\alpha$ is a 2-dimensional abelian subspace of $\g{g}_\alpha$.\label{th:main:a+n-v}
\end{enumerate}
\end{thmx}

We will show in Section~\ref{Sec:examples} that different choices of $\ell_\alpha$ or $\g{v}_\alpha$ above give rise to congruent foliations. 
We also determine the mean curvature of their leaves. A direct consequence of this computation is that

\begin{corollary}\label{cor:canonicalExtensionRH2}
	If $\mathcal{F}$ is a codimension two polar nonhyperpolar homogeneous foliation on $M$, then $\mathcal{F}$ is harmonic if and only if $\mathcal{F}$ is orbit equivalent to the canonical extension of the trivial action on a boundary component homothetic to the hyperbolic plane $\R \mathsf{H}^{2}$.
\end{corollary}

We say that a foliation $\mathcal{F}$ is harmonic if all of its leaves are minimal submanifolds of $M$. Equivalently, $\mathcal{F}$ is harmonic when the canonical projection from $M$ to the space of leaves of $\mathcal{F}$ is a harmonic map.

As a result of combining Theorem~\ref{th:MainTheoremList} with \cite{BerndtDiazTamaru10}, Corollary~\ref{corol:main} states the complete classification of homogeneous polar foliations of codimension two.
Recall that two roots $\alpha$, $\beta\in\Lambda$ are said to be disconnected if $\alpha+\beta$ is not a root (equivalently, if they are orthogonal).

\begin{corollary}\label{corol:main}
A codimension two homogeneous polar foliation on $M$ is orbit equivalent to the orbit foliation of a closed connected Lie group whose Lie algebra is:
\begin{enumerate}[\rm (a)]
\item $(\g{a}\ominus\g{v})\oplus\g{n}$, where $\g{v}$ is a $2$-dimensional subspace of $\g{a}$, or\label{corol:main:a-v+n}
\item $(\g{a}\ominus\ell)\oplus(\g{n}\ominus\ell_\alpha)$, where $\alpha\in\Lambda$ is a simple root, $\ell_\alpha$ is a line in $\g{g}_\alpha$, and $\ell$ is a line in $\ker\alpha$, or\label{corol:main:a-l+n-l}
\item $\g{a}\oplus\bigl(\g{n}\ominus(\ell_\alpha\oplus\ell_\beta)\bigr)$, where $\alpha$, $\beta\in\Lambda$ are orthogonal simple roots, and $\ell_\lambda$ is a line in $\g{g}_\lambda$, $\lambda\in\{\alpha,\beta\}$, or\label{corol:main:a+n-2l}
\item $(\ker\alpha)\oplus(\g{n}\ominus\ell_\alpha)$, where $\alpha\in\Lambda$ is a simple root, and $\ell_\alpha$ is a line in $\g{g}_\alpha$, or\label{corol:main:ker+n-l}
\item $\g{a}\oplus(\g{n}\ominus\g{v}_\alpha)$, where $\alpha\in\Lambda$ is a simple root, and $\g{v}_\alpha$ is a 2-dimensional abelian subspace of $\g{g}_\alpha$.\label{corol:main:a+n-v}
\end{enumerate}
\end{corollary}

Examples~(\ref{corol:main:a-v+n}) to~(\ref{corol:main:a+n-2l}) of Corollary~\ref{corol:main} are hyperpolar.

\begin{remark}
	The assumption that $M$ is irreducible implies that all the $G$-invariant metrics on $M$ are rescalings of the metric induced by the Killing form.
	If the space $M$ is reducible, then the results presented above still hold if we impose the condition that the metric on $M$ is (homothetic to) the one induced by the Killing form.
\end{remark}

The trivial action is always polar and the whole space is a section of this action.
On the other hand, cohomogeneity one polar actions are automatically hyperpolar.
An easy consequence of Theorem~\ref{th:MainTheoremList} is that the action corresponding to case~(\ref{th:main:ker+n-l}) exists and is nontrivial unless $\Sigma^+=\{\alpha\}$ and $\dim\g{g}_\alpha=1$. Thus,

\begin{corollary}
If $M$ is an irreducible symmetric space of noncompact type where all polar actions are hyperpolar, then $M$ is the real hyperbolic space $\R \mathsf{H}^2$.
\end{corollary}

This contrasts sharply with the situation in the compact setting: polar actions on irreducible symmetric spaces of compact type and higher rank are always hyperpolar.
This follows from a series of papers by Kollross that concluded in the paper~\cite{Kollross-Lytchak} by Kollross and Lytchak.

We now describe the structure of this paper. In Section~\ref{Sec:prelims} we review the basic theory of real semisimple Lie algebras and symmetric spaces that is needed for our purposes. In Section~\ref{Sec:examples}, we present the list of codimension two polar nonhyperpolar homogeneous foliations that appear in Theorem~\ref{th:MainTheoremList}. 
We also determine the curvature of their sections, the extrinsic geometry of their orbits, and prove Theorem~\ref{th:MainTheoremCanonicalExtension}.
Finally, Section~\ref{Sec:proof} contains the proof of Theorem~\ref{th:MainTheoremList}.

\section{Preliminaries}\label{Sec:prelims}

In this section we introduce the concepts, notations, and preliminary results that are used throughout this paper. We follow~\cite{Helgason} for the theory on symmetric spaces, and~\cite{Knapp} for semisimple Lie algebras.
Since symmetric spaces of noncompact type are Hadamard manifolds, another interesting source of information is~\cite{Eberlein}.

\subsection{Semisimple Lie algebras and symmetric spaces}\label{sec:roots}\hfill

A connected Riemannian symmetric space of noncompact type can be represented as a quotient $M=G/K$, where $(G,K)$ is a symmetric pair.
The group $G$ acts almost effectively on $M$, and $K$ is taken to be the isotropy group of $G$ at a point $o\in M$ that we fix from now on.
Thus, the group $G$ is semisimple with finite center and $K$ is a maximal compact subgroup of $G$.
The Lie algebra of $G$ is denoted by $\g{g}$.
The Killing form $\mathcal{B}$ of $\g{g}$ is nondegenerate because $\g{g}$ is a real semisimple Lie algebra; 
$\mathcal{B}$ is known to be negative definite on $\g{k}$, the Lie algebra of $K$.
Let $\g{p}$ denote the orthogonal complement of~$\g{k}$ in~$\g{g}$ with respect to the Killing form.
Then, $\g{g}=\g{k}\oplus\g{p}$ is a Cartan decomposition of $\g{g}$, and $\g{p}$ can be identified with the tangent space $T_o M$. The Killing form $\mathcal{B}$ restricted to $\g{p}$ is positive definite.
Let $\theta$ be the Cartan involution associated with the previous Cartan decomposition, that is, $\theta\vert_{\g{k}}=\id_{\g{k}}$, and $\theta\vert_{\g{p}}=-\id_{\g{p}}$.
The equation $\langle X,Y\rangle=-\mathcal{B}(X,\theta Y)$, $X$,~$Y\in\g{g}$, defines a positive definite inner product on $\g{g}$ that will be used extensively from now on.
We normalize the metric on $M$ so that its restriction to $\g{p}\times\g{p}$ is precisely the inner product defined above.

As a matter of notation, if $\g{v}$ and $\g{w}$ are subspaces of $\g{g}$ we denote the orthogonal complement of $\g{w}$ in $\g{v}$ as
\[
\g{v}\ominus\g{w}=\{X\in\g{v}:\langle X,Y\rangle=0,\text{ for all $Y\in\g{w}$}\}.
\]
We also denote by $\g{v}_\g{w}$ the orthogonal projection of $\g{v}$ onto $\g{w}$.
The norm of $X\in\g{g}$ is denoted by $\lvert X\rvert$.

If $F\colon\g{g}\to\g{g}$ is a linear map, its adjoint is the map $F^*\colon\g{g}\to\g{g}$ satisfying $\langle F(X),Y\rangle=\langle X,F^*(Y)\rangle$ for any $X$, $Y\in\g{g}$.
Recall that, for $g\in G$, conjugation by $g$ is denoted by $I_g\colon G\to G$, $h\mapsto ghg^{-1}$.
The map $\Ad\colon G\to\mathsf{GL}(\g{g})$, $g\mapsto\Ad(g)=I_{g*}$, is called the adjoint representation of $\g{g}$,
and the differential of $\Ad$ is
the adjoint map $\ad\colon\g{g}\to\g{gl}(\g{g})$ given by $\ad(X)(Y)=[X,Y]$.
Since $\Ad$ is a homomorphism of Lie groups and $\ad$ is its differential, they are related by the Lie exponential map $\Exp\colon\g{g}\to G$ and the matrix exponential map of $\mathsf{GL}(\g{g})$ as $\Ad(\Exp(X))=e^{\ad(X)}$, $X\in\g{g}$.
Then, it follows that
\begin{align*}
\ad(X)^*&{}=-\ad(\theta X),&
\Ad(\Exp(X))^*&{}=e^{-\ad(\theta X)}.
\end{align*}
Note that $\ad(X)$ is skew-adjoint if $X\in\g{k}$ and self-adjoint if $X\in\g{p}$.

We choose a maximal abelian subspace $\g{a}$ of $\g{p}$.
Any two maximal abelian subspaces of $\g{p}$ are conjugate by an element of $K$;
in particular they have the same dimension, which is called the rank of $M$, denoted by $r=\dim\g{a}$.
If $\g{a}^*$ denotes the dual vector space of $\g{a}$, for each $\lambda\in\g{a}^*$ we define
\[
\g{g}_\lambda=\{X\in\g{g}:\ad(H)X=\lambda(H)X,\text{ for all $H\in\g{a}$}\}.
\]
If $\lambda\neq 0$ and $\g{g}_\lambda\neq\{0\}$ then $\lambda$ is called a restricted root (or simply root).
We denote by $\Sigma$ the set of restricted roots.
Since $\g{a}$ is abelian, $\ad(\g{a})$ is a commuting family of self-adjoint operators.
The corresponding simultaneous diagonalization
\[
\g{g}=\g{g}_0\oplus\Bigl(\bigoplus_{\lambda\in\Sigma}\g{g}_\lambda\Bigr),
\]
is called the restricted root space decomposition of $\g{g}$  determined by~$\g{a}$.
Here, we have $\g{g}_0=\g{k}_0\oplus\g{a}$, where $\g{k}_0=\g{g}_0\cap\g{k}=Z_{\g{k}}(\g{a})$ is the centralizer of $\g{a}$ in $\g{k}$.
Moreover, $\theta\g{g}_\lambda=\g{g}_{-\lambda}$ and $[\g{g}_\lambda,\g{g}_\mu]\subseteq\g{g}_{\lambda+\mu}$ for any $\lambda$, $\mu\in\Sigma$.

We also use the following notation: for $\lambda\in\Sigma$ we denote by $H_\lambda\in\g{a}$ the metric dual of $\lambda$, which is defined as $\langle  H_\lambda, H\rangle=\lambda(H)$, $H\in\g{a}$.
The inner product on $\g{g}$ induces an inner product on $\g{a}$ by restriction, and on $\g{a}^*$ by duality as $\langle\lambda,\mu\rangle=\langle H_\lambda,H_\mu\rangle$, $\lambda$,~$\mu\in\Sigma$.

The set $\Sigma$ is a (nonreduced) root system.
One can introduce a notion of positivity in $\Sigma$ as follows:
let $H_{0}$ be an element of $\g{a}\setminus \bigcup_{\lambda\in \Sigma}\ker \lambda$ (we say in this case that $H_{0}$ is a regular element).
Then, a root $\lambda$ is said to be positive (respectively, negative) if $\lambda(H_{0})$ is positive (respectively, negative).
Furthermore, a positive root $\alpha$ is simple if it cannot be written as a sum of positive roots.
Let $\Sigma^{+}$ be the set of positive roots, and $\Lambda\subseteq \Sigma^{+}$ the set of simple roots.
We note that any two choices of $\Sigma^{+}$ (and thus of $\Lambda$) are conjugate by an element of $N_{K}(\g{a})$, the normalizer of $\g{a}$ in $K$.
The cardinality of $\Lambda$ is equal to the rank of $\g{g}$, and $\Lambda$ is a basis of the vector space $\g{a}$.
We define $\g{n}=\bigoplus_{\lambda\in\Sigma^+}\g{g}_\lambda$, which is a nilpotent subalgebra of $\g{g}$.
Then, $\g{g}=\g{k}\oplus\g{a}\oplus\g{n}$ is a direct sum of vector spaces called the Iwasawa decomposition of~$\g{g}$.
The Lie algebra $\g{a}\oplus\g{n}$ is solvable and its derived Lie algebra is~$\g{n}$.
If $A$ and $N$ denote the connected subgroups of $G$ whose Lie algebras are $\g{a}$ and $\g{n}$, respectively, the map $K\times A\times N\to G$, $(k,a,n)\mapsto kan$, is a diffeomorphism. The corresponding decomposition $G=KAN$ is called the Iwasawa decomposition of~$G$.
The connected subgroup $AN$ whose Lie algebra is $\g{a}\oplus\g{n}$ is simply connected, solvable, and acts simply transitively on $M$.
Therefore, $M$ is isometric to the Lie group $AN$ endowed with a left-invariant Riemannian metric.
Moreover, the tangent space $T_o M$ can be identified with $\g{a}\oplus\g{n}$. The Lie exponential map $\Exp\colon\g{a}\oplus\g{n}\to AN$ is a diffeomorphism.

Since $AN$ acts simply transitively on $M$, we may endow $AN$ with a left invariant metric $\langle \cdot, \cdot \rangle_{AN}$ such that $M$ and $AN$ are isometric via the map $g\in AN \mapsto g\cdot o$.
The differential of this map at $e$ is precisely the orthogonal projection of $\g{a}\oplus\g{n}$ onto $\g{p}$, and the metric in $AN$ satisfies $\langle X,Y \rangle_{AN}=\langle X_{\g{a}},Y_{\g{a}}\rangle + \frac{1}{2}\langle X_{\g{n}},Y_{\g{n}} \rangle$ for every $X$, $Y\in\g{a}\oplus\g{n}$.
Furthermore, the Levi-Civita connection of $AN$ is given in terms of left invariant vector fields as $4\langle \nabla_{X}Y,Z  \rangle_{AN}=\langle [X,Y]+(1-\theta)[\theta X,Y], Z\rangle$, where $X$, $Y$, $Z\in\g{a}\oplus\g{n}$.

Note that for any $X\in\g{g}$ we have $2X=(1+\theta)X+(1-\theta)X$, where $(1+\theta)X\in\g{k}$ and $(1-\theta)X\in\g{p}$.
We use the following notation in what follows:
\begin{align*}
\g{k}_\lambda&{}=(1+\theta)\g{g}_\lambda=(1+\theta)\g{g}_{-\lambda}=\g{k}\cap(\g{g}_\lambda\oplus\g{g}_{-\lambda}),\\
\g{p}_\lambda&{}=(1-\theta)\g{g}_\lambda=(1-\theta)\g{g}_{-\lambda}=\g{p}\cap(\g{g}_\lambda\oplus\g{g}_{-\lambda}).
\end{align*}
Obviously $\g{k}_\lambda=\g{k}_{-\lambda}$, $\g{p}_\lambda=\g{p}_{-\lambda}$, and $\g{k}_\lambda\oplus\g{p}_\lambda=\g{g}_\lambda\oplus\g{g}_{-\lambda}$ for each $\lambda\in\Sigma^+$.
For $X$, $Y\in\g{g}_\lambda$, we also use the equality
$(1-\theta)[\theta X,Y]=2\langle X,Y\rangle H_\lambda$.

Let $\Lambda=\{\alpha_1,\dots,\alpha_r\}$ be the set of simple roots, and $\{H^1,\dots,H^r\}\subseteq\g{a}$ its dual basis, that is, $\alpha_i(H^j)=\delta_i^j$ is the Kronecker delta.
We define $H^\Lambda=\sum_{i=1}^r H^i$.
If $\lambda\in\Sigma$, then $\lambda=\sum_{i=1}^r c_i\alpha_i$, where the $c_i$ are integers and $c_i\geq 0$ for all $i\in\{1,\dots,r\}$ if $\lambda$ is positive, or $c_i\leq 0$ for all $i$ if $\lambda$ is negative.
The integer $\lambda(H^\Lambda)=\sum_{i=1}^r c_i$ is called the level of the root $\lambda$.
This determines a gradation of $\g{g}$ as
\[
\g{g}=\bigoplus_{k\in\mathbb{Z}}\g{g}^k,\text{ where }
\g{g}^k=\bigoplus_{\substack{\lambda\in\Sigma\\ \lambda(H^\Lambda)=k}}\g{g}_\lambda.
\]
Then, $\theta\g{g}^k=\g{g}^{-k}$, $k\in\mathbb{Z}$, and $\g{g}^0=\g{g}_0$.
According to~\cite{KaneyukiAsano88}, we have $\g{g}^{k+1}=[\g{g}^1,\g{g}^k]$, $\g{g}^{-k-1}=[\g{g}^{-1},\g{g}^{-k}]$, $k\geq 1$.
We also set $\g{n}^k=\g{g}^k$, $k\geq 1$. Thus, $\g{n}$ is generated by $\g{n}^1$.
We also define $\g{p}^k=\g{p}\cap(\g{g}^k\oplus\g{g}^{-k})$.
Finally, there is a highest root of the root space decomposition of $\g{g}$.
Let $m$ denote its level.
In fact we have $\g{g}=\bigoplus_{k=-m}^m\g{g}^k$, $\g{n}=\bigoplus_{k=1}^m\g{n}^k$, and $\g{p}=\g{a}\oplus\bigl(\bigoplus_{k=1}^m\g{p}^k\bigr)$.

We define the element 
\[ 
\delta=\frac{1}{2}\sum_{\lambda\in\Sigma^{+}}(\dim \g{g}_{\lambda})\lambda.
\]
If $s_{\alpha}\colon \g{a}^{*}\to\g{a}^{*}$ is the root reflection with respect to  a simple root $\alpha\in \Lambda$, it follows from~\cite[Theorem 6.57]{Knapp} that $s_{\alpha}$ is induced by an element of $N_{K}(\g{a})$, the normalizer of $\g{a}$ in $K$.
In particular, $\dim \g{g}_{s_{\alpha}(\lambda)}=\dim \g{g}_{\lambda}$ for every $\lambda\in\Sigma$.
On the other hand, by~\cite[Lemma 2.61]{Knapp} $s_{\alpha}$ permutes all positive roots linearly independent from $\alpha$, while sending $\alpha$ to $-\alpha$.
As a consequence, $s_{\alpha}(\delta)=\delta-(\dim \g{g}_{\alpha})\alpha-2(\dim \g{g}_{2\alpha})\alpha$.
Taking the inner product with $\alpha$ yields
\begin{equation}\label{eq:delta}
	2\langle \delta,\alpha \rangle=\lvert\alpha\rvert^{2}(\dim \g{g}_{\alpha}+2\dim \g{g}_{2\alpha}).
\end{equation}

\subsection{Parabolic subalgebras and canonical extensions}\label{sec:parabolic}\hfill

Let $\Phi$ be a subset of simple roots. 
We denote by $\Sigma_\Phi$ the subset of $\Sigma$ spanned by $\Phi$, and by $\Sigma_\Phi^+$ the corresponding set of positive roots inside $\Sigma_\Phi$.
Then $\Sigma_\Phi$ is a root subsystem of $\Sigma$ and $\Phi$ is a simple system for $\Sigma_\Phi$.
We define
\begin{align*}
\g{a}_{\Phi}&{}=\bigcap_{\alpha\in\Phi}\ker\alpha,&
\g{l}_{\Phi}&{}=\g{g}_{0}\oplus\bigoplus_{\lambda\in\Sigma_{\Phi}}\g{g}_{\lambda},&
\g{n}_{\Phi}&{}=\bigoplus_{\lambda\in\Sigma^{+}\setminus\Sigma_{\Phi}^{+}}\g{g}_{\lambda}.
\end{align*}
Then, $\g{a}_\Phi$ is abelian, $\g{l}_\Phi$ is reductive, and $\g{n}_\Phi$ is nilpotent.
Moreover, $[\g{l}_\Phi,\g{n}_\Phi]\subseteq\g{n}_\Phi$.

By definition, $\g{q}_\Phi=\g{l}_\Phi\oplus\g{n}_\Phi$ is called the parabolic subalgebra of $\g{g}$ determined by $\Phi$~\cite{BorelJi2006}.
Let $\g{m}_\Phi=\g{l}_\Phi\ominus\g{a}_\Phi$ and $\g{g}_\Phi=[\g{l}_\Phi,\g{l}_\Phi]=[\g{m}_\Phi,\g{m}_\Phi]$.
It is known that $\g{g}_\Phi$ is semisimple, and that $\g{m}_\Phi$ normalizes $\g{a}_\Phi\oplus\g{n}_\Phi$.
The decomposition $\g{q}_\Phi=\g{m}_\Phi\oplus\g{a}_\Phi\oplus\g{n}_\Phi$ is called the Langlands decomposition of $\g{q}_\Phi$.

We denote by $Q_\Phi$, $M_\Phi$, $G_\Phi$, $A_\Phi$, and $N_\Phi$ the connected subgroups of $G$ whose Lie algebras are $\g{q}_\Phi$, $\g{m}_\Phi$, $\g{g}_\Phi$, $\g{a}_\Phi$, and $\g{n}_\Phi$, respectively.
The subgroup $Q_\Phi$ is the parabolic subgroup of $G$ associated with $\Phi$, and $Q_\Phi=M_\Phi A_\Phi N_\Phi$ is the corresponding Langlands decomposition of $Q_\Phi$ at Lie group level.

We define $B_\Phi=M_\Phi\cdot o=G_\Phi\cdot o$.
Then, $B_\Phi$ is a totally geodesic submanifold of $M$, and since $G_\Phi$ is semisimple, it is a symmetric space of noncompact type and of rank $\lvert\Phi\rvert$.
This submanifold is called the boundary component associated with $\Phi$.

Let $H_\Phi$ be a subgroup of the isometry group of $B_\Phi$.
Since $G_\Phi$ contains the isometry group of $B_\Phi$, it follows that $H_\Phi\subseteq G_\Phi\subseteq M_\Phi$.
As $M_\Phi$ normalizes $A_\Phi N_\Phi$, $H=H_\Phi A_\Phi N_\Phi$ is a subgroup of $G$ acting isometrically on $M$, and its Lie algebra is $\g{h}_{\Phi}\oplus\g{a}_{\Phi}\oplus\g{n}_{\Phi}$.
The action of $H$ on $M$ is called the canonical extension to $M$ of the action of $H_\Phi$ on $B_\Phi$~\cite{BerndtDiazTamaru10}.

\subsection{Maximal solvable subalgebras}\label{sec:Borel}\hfill

We say that a subalgebra $\g{b}\subseteq \g{g}$ is a Borel subalgebra if it is a maximal solvable subalgebra of $\g{g}$.
On the other hand, a subalgebra $\g{h}\subseteq \g{g}$ is a Cartan subalgebra if its complexification $\g{h}\otimes \mathbb{C}$ is a Cartan subalgebra of the complex semisimple Lie algebra $\g{g}\otimes\mathbb{C}$.
In particular, $\g{h}$ is abelian~\cite[Proposition~2.10]{Knapp}.
Note, however, that Cartan subalgebras of real semisimple Lie algebras are not necessarily conjugate.

Any Cartan subalgebra $\g{h}$ of $\g{g}$ is conjugate to a $\theta$-stable subalgebra~\cite[Proposition~6.59]{Knapp}.
Thus, we can assume that $\theta\g{h}=\g{h}$, which means that $\g{h}$ splits as a direct sum $\tilde{\g{t}}\oplus\tilde{\g{a}}$, where $\tilde{\g{t}}\subseteq \g{k}$ and $\tilde{\g{a}}\subseteq\g{p}$.
Both $\tilde{\g{t}}$ and $\tilde{\g{a}}$ are abelian subspaces of $\g{g}$.
In this case, $\dim\tilde{\g{t}}$ is called the compact dimension of $\g{h}$, and $\dim\tilde{\g{a}}$ is called the noncompact dimension of $\g{h}$.
The subalgebra $\tilde{\g{t}}$ is called the torus part of $\g{h}$, and $\tilde{\g{a}}$ is called the vector part of $\g{h}$.
We have that $\tilde{\g{a}}$ induces a root space decomposition on~$\g{g}$.
Note that, in principle, $\tilde{\g{a}}$ does not have to be a maximal abelian subspace of $\g{p}$ in this case.
Root spaces are defined analogously: for each $\tilde{\lambda}\in\tilde{\g{a}}^{*}$, let
\[
\tilde{\g{g}}_{\tilde{\lambda}}=\{ X\in\g{g}\colon \ad(H)X=\tilde{\lambda}(H)X \text{ for all $H\in\tilde{\g{a}}$} \},
\]
and define $\tilde{\Sigma}$ to be the set of all $\tilde{\lambda}\in\tilde{\g{a}}^{*}$ such that $\tilde{\lambda}\neq 0$ and $\tilde{\g{g}}_{\tilde{\lambda}}\neq 0$.
Since the family $\ad(\tilde{\g{a}})$ consists again of commuting self-adjoint endomorphisms, it follows that $\g{g}=\tilde{\g{g}}_{0}\oplus\bigl(\bigoplus_{\tilde{\lambda}\in\tilde{\Sigma}}\tilde{\g{g}}_{\tilde{\lambda}}\bigr)$.
Observe that $\tilde{\g{t}}\subseteq \tilde{\g{g}}_{0}\cap\g{k}$ since $\g{h}$ is abelian.
	
We now relate the previous decomposition to the root space decomposition induced by a maximal abelian subspace $\g{a}\subseteq \g{p}$ containing $\tilde{\g{a}}$.
Let $\Sigma'\subseteq \Sigma$ be the subset of roots that annihilate $\tilde{\g{a}}$. We then have the following equalities:

\begin{align*}
\tilde{\g{a}}&{}=\bigcap_{\lambda\in\Sigma'}\ker\lambda,& \tilde{\g{g}}_{0}&{}=\g{g}_{0}\oplus\Bigl(\bigoplus_{\lambda\in\Sigma'}\g{g}_{\lambda}\Bigr),& \tilde{\g{g}}_{\lambda}&{}=\bigoplus_{\substack{\lambda \in \Sigma\\ \lambda\vert_{\tilde{\g{a}}}=\tilde{\lambda}}}\g{g}_{\lambda}.
\end{align*}
Since $\Sigma'$ is an abstract root system in $(\g{a}\ominus\tilde{\g{a}})^{*}$, we may give a notion of positivity on $\Sigma$ that is compatible with that of $\Sigma'$ and $\tilde{\Sigma}$, that is, $\lambda\in\Sigma$ is positive if and only if $\lambda\in \Sigma'$ is positive or $\lambda\in\Sigma\setminus\Sigma'$ and $\lambda\vert_{\tilde{\g{a}}}\in\tilde{\Sigma}^{+}$.

\begin{remark}
For the sake of completeness we explain here how these notions of positivity can be made compatible.
One can define a notion of positivity on $\tilde{\Sigma}$ by fixing a regular element $\tilde{H}_{0}\in\tilde{\g{a}}$ (that is, $\tilde{\lambda}(\tilde{H}_{0})\neq 0$ for all $\tilde{\lambda}\in\tilde{\Sigma}$) and declaring $\tilde{\lambda}\in\tilde{\Sigma}$ to be positive if $\tilde{\lambda}(\tilde{H}_{0})>0$.
We also take a regular element $H'\in\g{a}\ominus\tilde{\g{a}}$ and define $\lambda\in\Sigma'$ to be positive whenever $\lambda(H')>0$. We now define $H_{0}=\tilde{H}_{0}+\varepsilon H'$, where $\varepsilon$ is a positive constant. Note that for every $\lambda\in \Sigma'$, $\lambda(H_{0})=\varepsilon \lambda(H')$, so $\lambda(H_{0})$ and $\lambda(H')$ have the same sign. Furthermore, if $\lambda\in\Sigma\setminus\Sigma'$, we have $\lambda(H_{0})=\lambda\vert_{\tilde{\g{a}}}(\tilde{H}_{0})+\varepsilon\lambda(H')$. Since the set of roots is finite, we can choose $\varepsilon>0$ sufficiently small so that $\lambda(H_{0})$ and $\lambda\vert_{\tilde{\g{a}}}(\tilde{H}_{0})$ have the same sign for all $\lambda\in\Sigma\setminus\Sigma'$.
\end{remark}

By~\cite[Theorem 4.1]{Mostow}, any Borel subalgebra $\g{b}$ of $\g{g}$ is of the form $\g{b}=\tilde{\g{t}}\oplus\tilde{\g{a}}\oplus\tilde{\g{n}}$ for an adequate choice of a Cartan subalgebra $\g{h}=\tilde{\g{t}}\oplus\tilde{\g{a}}$, a set of positive elements $\tilde{\Sigma}^{+}\subseteq\tilde{\Sigma}$, and where $\tilde{\g{n}}=\bigoplus_{\tilde{\lambda}\in\tilde{\Sigma}^{+}}\tilde{\g{g}}_{\tilde{\lambda}}$. We aim to restate this description of $\g{b}$ directly in terms of the root system induced by $\g{a}$.
	
We consider the subset $\Phi\subseteq \Sigma'$ of simple roots associated with the positivity criterion in $\Sigma'$. Note that $\Phi\subseteq \Lambda$. Indeed, by the construction of our set of positive roots in $\Sigma$, we have $\Phi\subseteq \Sigma^{+}$. Suppose $\alpha\in\Phi$ is not simple, so that $\alpha=\beta+\gamma$ for two positive roots $\beta$,~$\gamma\in\Sigma^{+}$. Since $\alpha$ is simple in $\Sigma'$, we have that $\beta$ and $\gamma$ cannot be simultaneously in $\Sigma'$, and combining this with the equation $0=\alpha(H_{0})=\beta(H_{0})+\gamma(H_{0})$, we deduce that neither $\beta$ nor $\gamma$ are in $\Sigma'$, and either $\beta$ or $\gamma$ is negative, a contradiction.
	
To summarize, we have found a subset $\Phi\subseteq\Lambda$ of simple roots for which $\Sigma'=\Sigma_{\Phi}$ is the root system generated by $\Phi$ and the following identities hold: $\tilde{\g{a}}=\g{a}_{\Phi}$, $\tilde{\g{g}}_{0}=\g{l}_{\Phi}$ and $\tilde{\g{n}}=\g{n}_{\Phi}$.
	
We have thus arrived at the following result.
	
\begin{theorem}\label{thm:MostowSimpleRoots}
Let $\g{g}$ be a real semisimple Lie algebra and $\g{b}$ a Borel subalgebra of~$\g{g}$.
Then $\g{b}$ contains a Cartan subalgebra $\g{h}$. Furthermore, there exists a Cartan decomposition $\g{g}=\g{k}\oplus\g{p}$, a maximal abelian subspace $\g{a}\subseteq\g{p}$, a choice of simple roots $\Lambda\subseteq \Sigma$, and a set $\Phi\subseteq\Lambda$ such that $\g{b}=\tilde{\g{t}}\oplus\g{a}_{\Phi}\oplus\g{n}_{\Phi}$, where $\tilde{\g{t}}$ is an abelian subspace of $\g{k}_{\Phi}=\g{k}\cap\g{l}_{\Phi}=\g{k}_{0}\oplus\bigl(\bigoplus_{\lambda\in\Sigma_{\Phi}^{+}}\g{k}_{\lambda}\bigr)$.
\end{theorem}
	
We say that a Cartan subalgebra $\g{h}$ (resp.\ Borel subalgebra~$\g{b}$) is maximally compact if its compact dimension is maximal, and maximally noncompact if its noncompact dimension is maximal.
Since $\tilde{\g{a}}$ is abelian in $\g{p}$, we have that $\g{h}$ (resp.~$\g{b}$) is maximally noncompact if and only if $\tilde{\g{a}}$ is a maximal abelian subspace of $\g{p}$ \cite[Proposition~6.47]{Knapp}. 

If a Borel subalgebra corresponds to a maximally noncompact Cartan subalgebra, then
$\Phi=\emptyset$ is the empty set, $\tilde{\g{a}}={\g{a}_\emptyset}=\g{a}$ is a maximal abelian subspace of $\g{p}$,
and $\g{t}=\tilde{\g{t}}$ is a maximal abelian subspace of $\g{k}_0$~\cite[Proposition~6.47 and~Lemma~6.62]{Knapp}.
This implies ${\g{n}_\emptyset}=\g{n}$, and thus, $\g{b}=\g{t}\oplus\g{a}\oplus\g{n}$.

\subsection{Polar actions}\hfill

Let $M$ be a Riemannian manifold and $H$ a connected Lie group acting on $M$ isometrically. Given any point $p\in M$, we denote by $H_{p}$ the isotropy subgroup at $p$, and by $H\cdot p$ the $H$-orbit through $p$.
We assume that the action of $H$ on $M$ is proper, that is, the map $H\times M\to M\times M$, $(g,p)\mapsto (p,g\cdot p)$ is proper.
In this case, isotropy groups are compact, the orbit space is Hausdorff, and the orbits of the action are closed and embedded submanifolds.
If the action is also effective, that is, the only element of the group acting trivially is the identity element, then $H$ can be assumed to be a closed subgroup of the isometry group of $M$ acting on $M$ in the natural way.
If $p\in M$ is any point, we define the isotropy representation of $H$ at $p$ as the map $g\in H_{p}\mapsto g_{*p}\in \mathsf{O}(T_{p}M)$, and the slice representation of $H$ at $p$ as its restriction to $\mathsf{O}(\nu_{p}(H\cdot p))$. Here, the notation $\nu_{p}N$ refers to the normal subspace of the submanifold $N\subseteq M$ at $p$.

Two isometric actions on $M$ are said to be orbit equivalent if there is an isometry of $M$ that maps the orbits of one action to the orbits of the other action.
They are said to be conjugate if there is an isometry of $M$ that is equivariant with respect to both actions.

We start by mentioning the following result, which is a refinement of~\cite[Lemma~2.5]{DiazDominguezKollross20}.

\begin{lemma}\label{lemma:OrbitEquivalenceNonClosed}
Let $M$ be a complete Riemannian manifold and $H$ and $\tilde{H}$ be connected, not necessarily closed, subgroups of the isometry group of $M$ such that $H\subseteq \tilde{H}$.
Suppose that there exists $o\in M$ such that $H\cdot o=\tilde{H}\cdot o$ is a closed subset of $M$, and the slice representation of $\tilde{H}$ at $o$ is trivial.
Then $H$ and $\tilde{H}$ act with the same orbits.
\end{lemma}

\begin{proof}
Let $p\in M$ be arbitrary. Since $H\cdot o$ is closed in $M$, we may find a point $q\in H\cdot o$ such that the distance from $q$ to $p$ is minimum among all points of $H\cdot o$. The first variation formula implies that the minimizing geodesic joining $q$ and $p$ must leave $H\cdot o$ perpendicularly. Thus, by homogeneity we may assume that $q=o$ and $p=\exp_{o}(\xi_{o})$, with $\xi_{o}\in\nu_{o}(H\cdot o)$. Let $\xi\in \Gamma(\nu(H\cdot o))$ be the unique $\tilde{H}$-equivariant vector field whose value at $o$ is $\xi_{o}$ (which exists because $\xi_{o}$ is fixed by the slice representation of $\tilde{H}$). Since $H\subseteq \tilde{H}$, $\xi$ is also the unique $H$-equivariant normal vector field along $H\cdot o$ generated by $\xi_{o}$. Due to~\cite[Section 2.1.8]{BerndtConsoleOlmos}, we obtain $H\cdot p = \{ \exp_{x}(\xi_{x}) \colon x\in H\cdot o \}=\tilde{H}\cdot p$.
\end{proof}

If $H$ is a closed subgroup of the isometry group of $M$, we say that the action of $H$ on $M$ is polar if there exists a submanifold $\Sigma$ in $M$ such that:
\begin{enumerate}[\rm (i)]
\item $\Sigma$ intersects all the orbits of $H$, and
\item if $p\in\Sigma$, then $T_p\Sigma$ and $T_p(H\cdot p)$ are orthogonal.
\end{enumerate}
The submanifold $\Sigma$ is called a section.
Sections are known to be totally geodesic (see for example~\cite{LorenzoSolonenko}).
If the section of a polar action is flat, then the action is called hyperpolar.

In this paper we assume that the action of $H$ on $M$ induces a foliation, that is, the action does not have singular orbits.
We say that $M$ is a Hadamard manifold if it is a simply connected complete Riemannian manifold with nonpositive sectional curvature.
Riemannian symmetric spaces of noncompact type are examples of Hadamard manifolds.
If follows from~\cite[Proposition~2.1]{BerndtDiazTamaru10} that all the orbits of $H$ are principal, that is, all isotropy groups are conjugate in $H$.

From~\cite[Theorem~4.1]{BerndtDiazTamaru10} we have the following criterion of polarity:

\begin{proposition}\label{th:criterion}
Let $M=G/K$ be a Riemannian symmetric space of noncompact type with Cartan decomposition $\g{g}=\g{k}\oplus\g{p}$.
Let $H\subseteq G$ be a closed subgroup acting on $M$ in such a way that its orbits form a foliation on $M$.
We define
\[
\g{h}_\g{p}^\perp=\{\xi\in\g{p}:\langle\xi,X\rangle=0,\text{ for all $X\in\g{h}$}\}.
\]
Then:
\begin{enumerate}[\rm (i)]
\item The group $H$ acts polarly on $M$ if and only if $\g{h}_\g{p}^\perp$ is a Lie triple system in $\g{p}$ and $[\g{h}_\g{p}^\perp,\g{h}_\g{p}^\perp]$ is orthogonal to $\g{h}$.
\item The group $H$ acts hyperpolarly on $M$ if and only if $\g{h}_\g{p}^\perp$ is an abelian subspace of $\g{p}$.
\end{enumerate}
Moreover, if $H_\g{p}^\perp$ denotes the subgroup of $G$ whose Lie algebra is $[\g{h}_\g{p}^\perp,\g{h}_\g{p}^\perp]\oplus\g{h}_{\g{p}}^{\perp}$, then $H_\g{p}^\perp\cdot o$ is a section of the action of $H$ on $M$.
\end{proposition}

The first criterion of polarity is credited to Gorodski~\cite{Gorodski04}.

Recall that a subspace $\g{v}$ of $\g{p}$ is called a Lie triple system if $[\g{v},[\g{v},\g{v}]]\subseteq\g{v}$.
In this case $[\g{v},\g{v}]\oplus\g{v}$ is a $\theta$-stable subalgebra of $\g{g}$ and the orbit through $o\in M$ of the connected subgroup of $G$ whose Lie algebra is $[\g{v},\g{v}]\oplus\g{v}$ is a totally geodesic submanifold of $M$.

\section{Examples of homogeneous polar foliations}\label{Sec:examples}

We now introduce the two families of polar homogeneous foliations on $M=G/K$ whose section is homothetic to the hyperbolic plane, and describe their extrinsic geometry. We assume the notation used in Section~\ref{Sec:prelims}.

\begin{theorem}\label{thm:polarExamples}
	Let $M=G/K$ be a symmetric space of noncompact type and choose an Iwasawa decomposition $\g{g}=\g{k}\oplus\g{a}\oplus\g{n}$ of $\g{g}$.
	Let $\alpha\in\Lambda$ be a simple root, and consider the following subspaces of $\g{a}\oplus\g{n}$:
	\begin{enumerate}[\rm (i)]
		\item $\g{s}_{\xi}=(\g{a}\ominus\R H_{\alpha})\oplus(\g{n}\ominus \R\xi)$, where $\xi\in\g{g}_{\alpha}$ is a nonzero vector.
		\item $\g{s}_{\g{v}}=\g{a}\oplus(\g{n}\ominus\g{v})$, where $\g{v}\subseteq \g{g}_{\alpha}$ is an abelian plane inside $\g{g}_{\alpha}$.
	\end{enumerate}
	
	The subspaces $\g{s}_{\xi}$ and $\g{s}_{\g{v}}$ are Lie subalgebras of $\g{a}\oplus\g{n}$.
	The corresponding connected subgroups $S_{\xi}$, $S_{\g{v}}$ act polarly on $M$ inducing a codimension two foliation whose section is a totally geodesic $\R\mathsf{H}^{2}$ with constant curvature $-\lvert\alpha\rvert^{2}$.
\end{theorem}

\begin{proof}
It is clear that $\g{s}_{\xi}$ and $\g{s}_{\g{v}}$ are subalgebras of $\g{a}\oplus\g{n}$, so we can consider the connected Lie subgroups $S_{\xi}$, $S_{\g{v}}$ associated with these subalgebras.
Since $AN$ acts simply transitively on $M$ and $\Exp\colon \g{a}\oplus\g{n}\to AN$ is a diffeomorphism, it follows that $S_{\xi}$ and $S_{\g{v}}$ are closed subgroups inducing a homogeneous foliation on $M$ of codimension $2$.
It remains to show that both subgroups act polarly.

Firstly, if $S=S_{\xi}$, we see that $\g{s}_{\g{p}}^{\perp}=\operatorname{span}\{ H_{\alpha},(1-\theta)\xi \}$.
Now, a direct computation shows that $[\g{s}_{\g{p}}^{\perp},\g{s}_{\g{p}}^{\perp}]=\R (1+\theta)\xi$ is orthogonal to $\g{s}$, and $[\g{s}_{\g{p}}^{\perp},[\g{s}_{\g{p}}^{\perp},\g{s}_{\g{p}}^{\perp}]]$ is spanned by $[H_{\alpha},(1+\theta)\xi]=\lvert\alpha\rvert^{2}(1-\theta)\xi$ and $[(1-\theta)\xi,(1+\theta)\xi]=-(1-\theta)[\theta\xi,\xi]=-2\lvert\xi\rvert^{2}H_{\alpha}$.
We therefore obtain $[\g{s}_{\g{p}}^{\perp},[\g{s}_{\g{p}}^{\perp},\g{s}_{\g{p}}^{\perp}]]=\g{s}_{\g{p}}^{\perp}$, which means that $\g{s}_{\g{p}}^{\perp}$ is a Lie triple system.
By applying Proposition~\ref{th:criterion}, we deduce that $S_{\xi}$ acts polarly, as desired.
Note that if $S_{\g{p}}^{\perp}\cdot o$ is the section through $o$, then $S_{\g{p}}^{\perp}\cdot o$ is a closed, simply connected, totally geodesic surface whose tangent space is $\g{s}_{\g{p}}^{\perp}$. Its sectional curvature can be calculated using for example~\cite[Chapter IV,~Theorem~4.2]{Helgason}, which yields
\begin{align*}
\operatorname{sec}(S_{\g{p}}^{\perp}\cdot o)
&{}=\frac{-\langle[[H_{\alpha},(1-\theta)\xi],(1-\theta)\xi],H_{\alpha}\rangle}{\lvert(1-\theta)\xi\rvert^2\lvert H_\alpha\rvert^2} \\
&{}=\frac{-\lvert[H_{\alpha},(1-\theta)\xi]\rvert^{2}}{2\lvert\xi\rvert^{2}\lvert\alpha\rvert^{2}}
=\frac{-\lvert(1+\theta)\lvert\alpha\rvert^{2}\xi\rvert^{2}}{{2\lvert\xi\rvert^{2}\lvert\alpha\rvert^{2}}}
=-\lvert\alpha\rvert^{2},
\end{align*}
so $S_\g{p}^{\perp}\cdot o$ is a real space form of constant curvature $-\lvert\alpha\rvert^{2}$.

In the case $S=S_{\g{v}}$, the normal space is $\g{s}_{\g{p}}^{\perp}=(1-\theta)\g{v}$.
Choose two orthogonal vectors $\xi$, $\eta\in\g{v}$ with norm ${1}/{\sqrt{2}}$.
Since $[(1-\theta)\xi,(1-\theta)\eta]=-2[\xi,\theta\eta]\in\g{k}_{0}$, it follows that $[\g{s}_{\g{p}}^{\perp},\g{s}_{\g{p}}^{\perp}]=\R [\xi,\theta\eta]$ is perpendicular to $\g{s}$.
Furthermore, $\theta[\xi,\theta\eta]=[\xi,\theta\eta]$ yields 
\begin{align*}
[(1-\theta)\xi,[\xi,\theta\eta]]
&{}=(1-\theta)[\xi,[\theta\xi,\eta]]
=-(1-\theta)[\eta,[\xi,\theta\xi]]\\
&{}=(1-\theta)[\eta,\lvert\xi\rvert^{2}H_{\alpha}]
=-\lvert\xi\rvert^{2}\lvert\alpha\rvert^{2}(1-\theta)\eta\in\g{s}_{\g{p}}^{\perp},
\end{align*}
and a similar calculation gives $[(1-\theta)\eta,[\xi,\theta\eta]]=\lvert\alpha\rvert^{2}\lvert\eta\rvert^{2}(1-\theta)\xi\in\g{s}_{\g{p}}^{\perp}$, and thus $[\g{s}_{\g{p}}^{\perp},[\g{s}_{\g{p}}^{\perp},\g{s}_{\g{p}}^{\perp}]]=\g{s}_{\g{p}}^{\perp}$.
Proposition~\ref{th:criterion} readily implies that the action of $S_{\g{v}}$ is polar with section $S_{\g{p}}^{\perp}\cdot o$.
The same argument given in the previous paragraph allows us to determine the section by computing its curvature.
In this case, taking into account our previous calculations,
\[
\operatorname{sec}(S_{\g{p}}^{\perp}\cdot o)
=-\frac{\langle [[(1-\theta)\xi,(1-\theta)\eta],(1-\theta)\eta],(1-\theta)\xi \rangle}{\lvert(1-\theta)\xi\rvert^{2}\lvert(1-\theta)\eta\rvert^{2}}
=-\lvert\alpha\rvert^2,
\]
which finishes the proof.
\end{proof}

The previous theorem shows that the examples that appear in Theorem~\ref{th:MainTheoremList} give rise to homogeneous polar foliations.
Furthermore, it follows from the next lemma that different choices of $\xi$ in case~(\ref{th:main:ker+n-l}) or of $\g{v}$ in~(\ref{th:main:a+n-v}) give orbit equivalent actions.

\begin{lemma}
Let $\alpha\in\Lambda$ and $k\geq 1$.
Then, the group $K_0$ acts transitively on the set of abelian subspaces of dimension $k$ of $\g{g}_\alpha$.
\end{lemma}

\begin{proof}
Following~\cite[Chapter~IX, \S2]{Helgason}, we consider the Lie subalgebra $\g{g}^\alpha$ generated by $\g{g}_\alpha$ and $\g{g}_{-\alpha}$.
This Lie algebra is simple and its Cartan decomposition is $\g{g}^\alpha=\g{k}^\alpha\oplus\g{p}^\alpha$, with $\g{k}^\alpha=\g{k}\cap\g{g}^\alpha$, $\g{p}^\alpha=\g{p}\cap\g{g}^\alpha$.
It turns out that $\R H_\alpha$ is a maximal abelian subspace of $\g{p}^\alpha$, and
the root space decomposition of $\g{g}^\alpha$ is $\g{g}^\alpha=\g{g}_{-2\alpha}\oplus\g{g}_{-\alpha}\oplus(\g{k}_0^\alpha\oplus\R H_\alpha)\oplus\g{g}_\alpha\oplus\g{g}_{2\alpha}$, where $\g{k}_0^\alpha$ is the centralizer of $\R H_\alpha$ in $\g{k}^\alpha$, and $\g{k}_0^\alpha=\g{k}_0\cap\g{g}^\alpha$.
Let $G^\alpha$, $K^\alpha$, $K_0^\alpha$ be the connected subgroups of $G$ whose Lie algebras are $\g{g}^\alpha$, $\g{k}^\alpha$, and $\g{k}_0^\alpha$.
Then, by~\cite[Chapter~IX, Lemma~2.3]{Helgason}, we have
$K^\alpha=K\cap G^\alpha$ and $K_0^\alpha=K_0\cap G^\alpha$.
Therefore, in order to prove this lemma, it suffices to show that $K_0^\alpha$ acts transitively on the set of abelian subspaces of $\g{g}_\alpha$.

Obviously, $G^\alpha/K^\alpha$ is a Riemannian symmetric space of noncompact type and rank one, that is, a hyperbolic space $\mathbb{F}\mathsf{H}^n$, where $\mathbb{F}\in\{\R,\mathbb{C},\mathbb{H},\mathbb{O}\}$ and $n\geq 2$ ($n=2$ if $\mathbb{F}=\mathbb{O}$).
Note that $\g{g}_\alpha\cong \mathbb{F}^{n-1}$ and $\dim\g{g}_{2\alpha}=\dim_\R\mathbb{F}-1$.
If $\mathbb{F}=\R$, then $\g{g}_\alpha$ is abelian and $K_0^\alpha\cong\mathsf{SO}(n-1)$ acts in the standard way on $\g{g}_\alpha$; this action is transitive on the Grassmannian of $k$-planes of $\R^{n-1}$.
If $\mathbb{F}=\mathbb{O}$, then the only nonzero abelian subspaces of $\g{g}_\alpha\cong\mathbb{O}$ are 1-dimensional, and $K_0^\alpha\cong\mathsf{Spin}(7)$ acts on $\mathbb{O}$ by its irreducible 8-dimensional spin representation, which is transitive in $\mathsf{S}^{7}$ \cite{Borel}.
Finally, if $\mathbb{F}\in\{\mathbb{C},\mathbb{H}\}$, recall that abelian subspaces of $\g{g}_\alpha$ are precisely totally real subspaces of $\g{g}_\alpha\cong\mathbb{F}^{n-1}$.
In these cases we have the standard action of $\mathsf{S}(\mathsf{U}(n-1)\mathsf{U}(1))$ on $\mathbb{C}^{n-1}$ if $\mathbb{F}=\mathbb{C}$, and the standard action of $\mathsf{Sp}(n-1)\mathsf{Sp}(1)$ on $\mathbb{H}^{n-1}$ if $\mathbb{F}=\mathbb{H}$.
Thus, if $\g{v}_1$ and $\g{v}_2$ are two totally real subspaces of $\g{g}_\alpha$ of dimension $k$, choose an orthonormal basis of $\g{v}_1$ and an orthonormal basis of $\g{v}_2$.
Since $\g{v}_1$ and $\g{v}_2$ are totally real, these two bases are not only orthonormal, but $\mathbb{F}$-orthonormal.
By definition of $\mathsf{U}(n-1)$ or $\mathsf{Sp}(n-1)$ it is then clear that there is an element of $K_0^\alpha$ that maps one basis to the other.
This finishes the proof.
\end{proof}

We now exhibit examples~(\ref{th:main:ker+n-l}) and~(\ref{th:main:a+n-v}) of Theorem~\ref{th:MainTheoremList} as canonical extensions of actions on a rank one boundary component.
Let $\alpha\in\Lambda$ be a simple root, $\xi\in \g{g}_{\alpha}$ a unit vector and $\g{v}\subseteq \g{g}_{\alpha}$ an abelian plane.
We consider the set $\Phi=\{\alpha\}\subseteq \Lambda$.
Then, the subalgebras constructed in Section~\ref{sec:parabolic} take the form
\begin{equation*}
	\begin{split}
		\g{l}_{\Phi}={}&\g{g}_{0}\oplus\g{g}_{-2\alpha}\oplus\g{g}_{-\alpha}\oplus\g{g}_{\alpha}\oplus\g{g}_{2\alpha},\\
		\g{a}_{\Phi}={}&\ker \alpha,\\
		\g{n}_{\Phi}={}&\g{n}\ominus(\g{g}_{\alpha}\oplus\g{g}_{2\alpha}), \\
		\g{m}_{\Phi}={}& \g{k}_{0}\oplus\R H_{\alpha}\oplus\g{g}_{-2\alpha}\oplus\g{g}_{-\alpha}\oplus\g{g}_{\alpha}\oplus\g{g}_{2\alpha},
	\end{split}
\end{equation*}
and $B_{\Phi}=M_{\Phi}\cdot o$ is a rank one noncompact symmetric space whose tangent space at $o$ is $T_{o}B_{\Phi}=\R H_{\alpha}\oplus \g{p}_{\alpha}\oplus \g{p}_{2\alpha}$.
Consider the subalgebras $\hat{\g{s}}_{\xi}=(\g{g}_{\alpha}\ominus \R \xi)\oplus \g{g}_{2\alpha}$ and $\hat{\g{s}}_{\g{v}}=\R H_{\alpha} \oplus (\g{g}_{\alpha}\ominus \g{v})\oplus \g{g}_{2\alpha}$ of $\g{g}_{\Phi}=[\g{m}_{\Phi},\g{m}_{\Phi}]$.
The corresponding connected subgroups $\hat{S}_{\xi}$ and $\hat{S}_{\g{v}}$ act polarly on $B_{\Phi}$ inducing a foliation, due to Theorem~\ref{thm:polarExamples}.
Recall that the canonical extensions of the actions of $\hat{S}_{\xi}$ and $\hat{S}_{\g{v}}$ are the actions of the subgroups $\hat{S}_{\xi}A_{\Phi}N_{\Phi}$ and $\hat{S}_{\g{v}}A_{\Phi}N_{\Phi}$, respectively.
Observe that $\hat{\g{s}}_{\xi}\oplus\g{a}_{\Phi}\oplus\g{n}_{\Phi}=\g{s}_{\xi}$, while $\hat{\g{s}}_{\g{v}}\oplus\g{a}_{\Phi}\oplus\g{n}_{\Phi}=\g{s}_{\g{v}}$, and this readily implies that these canonical extensions are precisely the actions of $S_{\xi}$ and $S_{\g{v}}$.
We deduce that Theorem~\ref{th:MainTheoremList} implies Theorem~\ref{th:MainTheoremCanonicalExtension}.

The remaining part of this section will be devoted to computing the mean curvature of the orbits in each example.
To this end, we consider the solvable model $M=AN$ discussed in Section~\ref{sec:roots}.
If $\g{s}\subseteq \g{a}\oplus \g{n}$, we refer to its orthogonal complement in $\g{a}\oplus\g{n}$ with respect to $\langle\cdot,\cdot\rangle_{AN}$ as $\g{s}^{\perp}$.
Note that if $S\subseteq AN$ is a Lie subgroup, the isometry $g\in AN\mapsto g\cdot o\in M$ induces an orbit equivalence between the action of $S$ on $M$ and the action of $S$ on $AN$ by left multiplication.

Recall that if $M\subseteq \widetilde{M}$ is a submanifold of a Riemannian manifold with second fundamental form $\II$, we define the mean curvature vector of $M$ at $p$ as $\mathcal{H}_{p}=\sum_{i}\II(e_{i},e_{i})$, where $\{e_{i}\}_{i}$ is an orthonormal basis of $T_{p} M$.
In other words, $\mathcal{H}$ is the trace of the second fundamental form.
If $S\subseteq AN$ is a connected subgroup of $AN$, it is easy to see from the formula for the Levi-Civita connection that the second fundamental form of $S\subseteq AN$ at $e$ satisfies the identity 
\begin{equation}\label{eq:II}
	\langle \II(X,X),\eta \rangle_{AN}=\frac{1}{4}\langle (1-\theta)[\theta X,X],\eta \rangle
\end{equation}
for each $X\in \g{s}$ and $\eta\in \g{s}^{\perp}$.

Let us start by discussing foliations of type (\ref{th:main:ker+n-l}).

\begin{proposition}
	Let $\alpha\in\Lambda$ and $\xi\in \g{g}_{\alpha}$ a vector such that $\langle \xi,\xi\rangle=1$. All the orbits of $S_{\xi}$ are isometrically congruent.
	Furthermore, the mean curvature vector of $S_{\xi}$ at $e$ is given by the following expression:
	\[
		\mathcal{H}_{e}=\left( \dim \g{g}_{\alpha} + 2 \dim \g{g}_{2\alpha}-1 \right)H_{\alpha}.
	\]
\end{proposition}
\begin{proof}
	Observe that $\g{s}_{\xi}=\ker \alpha\oplus (\g{n}\ominus\R\xi)$ is an ideal of $\g{a}\oplus\g{n}$.
	As a consequence, if $g\in AN$ is an arbitrary point, we have $S_{\xi}\cdot g=gg^{-1}S_{\xi}g=gS_{\xi}$, because $S_{\xi}$ is a normal subgroup of $AN$.
	Thus, $S_{\xi}\cdot g$ is isometric to $S_{\xi}$ via the left multiplication by $g$.
	
	We proceed to compute $\mathcal{H}_{o}$.
	For this, it suffices to determine the vectors $\II(H,H)$ and $\II(X,X)$ for each $H\in \g{a}$ and $X\in \g{g}_{\lambda}\ominus \R \xi$, where $\lambda$ is any positive root.
	Given any $H\in \g{a}$, it is clear from \eqref{eq:II} that $\II(H,H)=0$.
	On the other hand, if $\lambda \in \Sigma^{+}$ and $X\in \g{g}_{\lambda}$ is a unit vector orthogonal to $\xi$, we obtain that $1=\langle X,X \rangle_{AN}=|X|^{2}/2$ and $\langle \II(X,X),\eta\rangle_{AN}=\frac{1}{2}\langle |X|^{2} H_{\lambda},\eta \rangle=\langle H_{\lambda},\eta \rangle=\langle \frac{\langle \lambda,\alpha \rangle}{|\alpha|^{2}}H_{\alpha},\eta \rangle_{AN}$ for each $\eta\in \g{s}^{\perp}=\operatorname{span}\{ H_{\alpha},\xi \}$, which means that $\II(X,X)=\frac{\langle \lambda,\alpha \rangle}{|\alpha|^{2}}H_{\alpha}$.
	In conclusion,
	\begin{align*}
	\mathcal{H}_{e}={}&\frac{1}{|\alpha|^{2}}\bigg( \sum_{\lambda \in \Sigma^{+}\setminus \{\alpha\}} (\dim \g{g}_{\lambda})\langle \lambda, \alpha \rangle+(\dim \g{g}_{\alpha}-1) |\alpha|^{2} \bigg)H_{\alpha}=\frac{1}{|\alpha|^{2}}(\langle 2\delta,\alpha \rangle-|\alpha|^{2})H_{\alpha} \\
	={}&(\dim \g{g}_{\alpha}+2 \dim \g{g}_{2\alpha}-1)H_{\alpha},
	\end{align*}
	where we have used \eqref{eq:delta} for the last equality.
\end{proof}

A direct consequence of the previous proposition is that the foliation induced by $S_{\xi}$ consists of congruent minimal submanifolds if and only if $2\alpha\notin\Sigma$ and $\dim \g{g}_{\alpha}=1$.
If this is the case, then $\g{s}_{\xi}=\g{a}_{\Phi}\oplus\g{n}_{\Phi}$ for $\Phi=\{\alpha\}$, and $B_{\Phi}$ is homothetic to the hyperbolic plane.
Furthermore, $\hat{\g{s}}_{\xi}=0$, which means that $\hat{S}_{\xi}$ acts trivially on $B_{\Phi}$.
More generally, the extrinsic geometry of the orbits for the solvable part $A_{\Phi}N_{\Phi}$ of a parabolic subgroup was studied by Tamaru in \cite{Tamaru}, where he proved that for any choice of $\Phi\subseteq\Lambda$ the subgroup  $A_{\Phi}N_{\Phi}$ induces a harmonic foliation on $M$ whose orbits are congruent.

We now consider the foliations from case (\ref{th:main:a+n-v}).
In this setting, the orbits are not isometrically congruent, as their mean curvature does not have constant length.
More precisely:

\begin{proposition}
	Let $\alpha\in\Lambda$ be a simple root and $\g{v}\subseteq \g{g}_{\alpha}$ an abelian subspace of dimension $2$. Fix a vector $\xi\in \g{v}$ with $|\xi|=1$, and denote by $\mathcal{H}_{t}$ the mean curvature vector of $S_{\g{v}}\cdot \Exp(t\xi)$ at $\Exp(t\xi)$ and by $L_{\Exp(t\xi)}\colon AN\to AN$ the left translation by $\Exp(t\xi)$. Then,
	\[
		(L_{\Exp(-t\xi)})_{*\Exp(t\xi)}\mathcal{H}_{t}=\frac{t|\alpha|^{2}}{2+t^{2}|\alpha|^{2}}(\dim \g{g}_{\alpha}+2 \dim \g{g}_{2\alpha} - 1)(tH_{\alpha}-2\xi).
	\]
	In particular, the orbit through $\Exp(t\xi)$ is minimal if and only if $t=0$.
\end{proposition}
\begin{proof}
	Firstly, if $g=\Exp(t\xi)\in AN$, we deduce that $S_{\g{v}}\cdot g=gg^{-1}S_{\g{v}}g=g(g^{-1}S_{\g{v}}g)$ is isometric to $g^{-1}S_{\g{v}}g$ by left translation.
	Thus, it suffices to compute the mean curvature $\tilde{\mathcal{H}}_{t}$ of $\tilde{S}=g^{-1}S_{\g{v}}g$ at $e$.
	To this end, we compute the Lie algebra $\tilde{\g{s}}=\Ad(g^{-1})\g{s}_{\g{v}}$ of $g^{-1}S_{\g{v}}g$.
	Observe that $\Ad(g^{-1})\g{s}_{\g{v}}\subseteq \g{a}\oplus\g{n}$ is orthogonal to $\Ad(g)^{*}\g{v}=e^{-t\ad(\theta\xi)}\g{v}$ with respect to the inner product $\langle \cdot,\cdot \rangle$.
	Given any $\eta\in \g{v}$, we have 
	\begin{align*}
		e^{-t \ad(\theta \xi)}\eta\equiv{}
		&\eta-t [\theta\xi,\eta]\pmod{\theta\g{n}}\\
		={}&\eta-t\langle \xi,\eta \rangle H_{\alpha} - \frac{t}{2}(1+\theta)[\theta\xi,\eta]\pmod{\theta\g{n}} \\
		\equiv{}&\eta-t \langle \xi,\eta \rangle H_{\alpha} \pmod{ \g{k}_{0}\oplus\theta\g{n} }.
	\end{align*}
	Therefore, if we consider an orthonormal basis $\{\xi,\eta\}$ of $\g{v}$, it is immediate that the orthogonal complement of $\Ad(g^{-1})\g{s}_{\g{v}}$ in $\g{a}\oplus\g{n}$ is $\operatorname{span}\{ t H_{\alpha}-\xi,\eta \}$.
	As a consequence, $\tilde{\g{s}}=\Ad(g^{-1})\g{s}_{\g{v}}=\ker\alpha\oplus(\g{n}\ominus\g{v})\oplus\R(H_{\alpha}+t|\alpha|^{2}\xi)$.
	The normal space $\g{s}^{\perp}$ is given by
	$\tilde{\g{s}}^{\perp}=\R\eta \oplus \R (tH_{\alpha}-2\xi)$.
	
	Assume $H\in \ker \alpha$. In this case, we directly have from \eqref{eq:II} that $\II(H,H)=0$.
	
	Now, suppose that $\lambda\in\Sigma^{+}$ and $X\in \g{g}_{\lambda}\ominus \g{v}$ is such that $1=\langle X,X \rangle_{AN}=\frac{1}{2}|X|^{2}$. Then, $\II(X,X)$ satisfies $\langle \II(X,X),\nu \rangle_{AN}=\frac{1}{2}\langle|X|^{2} H_{\lambda},\nu \rangle=\langle H_{\lambda},\nu \rangle=\langle \frac{t\langle \lambda, \alpha \rangle}{2+t^{2}|\alpha|^{2}}(t H_{\alpha}-2\xi),\nu \rangle_ {AN}$ for every $\nu\in \g{s}^{\perp}$, and thus $\II(X,X)=\frac{t\langle \lambda,\alpha \rangle}{2+t^{2}|\alpha|^{2}}(tH_{\alpha}-2\xi)$.
	
	Finally, consider the vector $Y=H_{\alpha}+t|\alpha|^{2}\xi$, whose norm squared is $\langle Y,Y \rangle_{AN}=|\alpha|^{2}+\frac{1}{2}t^{2}|\alpha|^{4}$.
	Note that $(1-\theta)[\theta Y,Y]=2t|\alpha|^{4}(tH_{\alpha}-(1-\theta)\xi)$, so we deduce from \eqref{eq:II} that $\II(Y,Y)=\frac{t|\alpha|^{4}}{2}(tH_{\alpha}-2\xi)$.
	As a consequence, the normalized vector $Z=Y/|Y|_{AN}$ satisfies $\II(Z,Z)=\frac{t|\alpha|^{2}}{2+t^{2}|\alpha|^{2}}(tH_{\alpha}-2\xi)$.
	From these calculations, we obtain that the mean curvature of $\tilde{S}$ at $o$ is given by
	\begin{equation*}
		\begin{split}
			\tilde{\mathcal{H}}_{t}={}&\frac{t}{2+t^{2}|\alpha|^{2}}\bigg( \sum_{\lambda\in \Sigma^{+}\setminus \{\alpha\}} (\dim \g{g}_{\lambda})\langle\lambda,\alpha\rangle + (\dim \g{g}_{\alpha}-1)|\alpha|^{2} \bigg)(tH_{\alpha}-2\xi)\\
			={}&\frac{t}{2+t^{2}|\alpha|^{2}}(\langle 2\delta,\alpha \rangle - |\alpha|^{2})(tH_{\alpha}-2\xi)\\
			={}&\frac{t|\alpha|^{2}}{2+t^{2}|\alpha|^{2}}(\dim \g{g}_{\alpha}+2 \dim \g{g}_{2\alpha} - 1)(tH_{\alpha}-2\xi).
		\end{split}
	\end{equation*}
	Finally, note that the existence of an abelian plane inside $\g{g}_{\alpha}$ implies that the integer $\dim \g{g}_{\alpha}+2\dim \g{g}_{2\alpha}-1$ is positive, so the orbit through $\Exp(t\xi)$ is minimal if and only if $t=0$, as desired. \qedhere
\end{proof}

In particular, the homogeneous foliation induced by $S_{\g{v}}$ is never harmonic independently of the choice of $\g{v}$.
From here, Corollary~\ref{cor:canonicalExtensionRH2} follows immediately.

\begin{corollary}
	No polar homogeneous foliation constructed as in case (\ref{th:main:ker+n-l}) of Theorem~\ref{th:MainTheoremList} is orbit equivalent to a homogeneous foliation given in case (\ref{th:main:a+n-v}).
\end{corollary}

\section{Proof of Theorem~\ref{th:MainTheoremList}}\label{Sec:proof}

Now we prove that the examples appearing in Theorem~\ref{th:MainTheoremList} are the only examples of codimension two homogeneous polar foliations on symmetric spaces of noncompact type.

From~\cite[Proposition~2.2]{BerndtDiazTamaru10} we obtain:

\begin{proposition}\label{th:solvable}
	Let $M$ be a Hadamard manifold, and let $H$ be a connected closed subgroup of the isometry group of $M$ acting on $M$ in such a way that the orbits of $H$ form a foliation.
	Then, all the orbits of $H$ are principal, and there is a connected closed solvable group $S$ acting isometrically on $M$ whose orbits coincide with the orbits of $H$.
\end{proposition}

Let $M=G/K$ be a symmetric space of noncompact type endowed with the metric induced by the Killing form.
We assume the notation introduced in Subsection~\ref{sec:roots}.
Thus, $K$ is the isotropy group at $o\in M$, we have a Cartan decomposition $\g{g}=\g{k}\oplus\g{p}$, a choice of maximal abelian subspace $\g{a}$ of $\g{p}$ that determines a root space decomposition $\g{g}=\g{g}_0\oplus\bigl(\bigoplus_{\lambda\in\Sigma}\g{g}_\lambda\bigr)$, and a positivity criterion that selects a set of positive roots $\Sigma^+$.
We denote by $\Lambda$ the set of simple roots.
We define $\g{n}=\bigoplus_{\lambda\in\Sigma^+}\g{g}_\lambda$, and recall that $\g{k}_0=\g{g}_0\cap\g{k}$.

Assume that $H$ is a connected closed subgroup of the isometry group $G$ that acts polarly on $M$, and that the orbits of $H$ on $M$ induce a foliation.
Proposition~\ref{th:solvable} says that there exists a solvable subgroup $S$ of $G$ whose orbits coincide with the orbits of $H$.
Let $\g{s}$ be the Lie algebra of $S$.
Then, $\g{s}$ is contained in a Borel subalgebra $\g{b}$ of $\g{g}$.
See Subsection~\ref{sec:Borel} for further details.
The next result states that we may assume that $\g{s}$ is contained in a maximally noncompact Borel subalgebra.

\begin{proposition}\label{th:Borel:tan}
The leaves of a homogeneous polar foliation on $M=G/K$ coincide, up to isometric congruence, with the orbits of a connected closed solvable subgroup $S$ of $G$ whose Lie algebra $\g{s}$ is contained in a maximally noncompact Borel subalgebra of the form $\g{t}\oplus\g{a}\oplus\g{n}$, where $\g{t}\subseteq \g{k}_{0}$ is abelian.
\end{proposition}

\begin{proof}
By Proposition~\ref{th:solvable} there is a solvable subgroup $S$ of $G$ whose orbits form the homogeneous polar foliation under investigation.
According to Proposition~\ref{thm:MostowSimpleRoots}, we may assume that $\g{s}$ is contained in a maximal solvable subalgebra of the form $\tilde{ \g{t}}\oplus\g{a}_{\Phi}\oplus\g{n}_{\Phi}$, with $\tilde{\g{t}}\subseteq \g{k}_{\Phi}$ an abelian subspace, and $\Phi\subseteq \Lambda$ a subset of simple roots.
In particular, the tangent space of $S\cdot o$ at $o$, as a subspace of $\g{p}$, is contained in $(1-\theta)(\g{a}_{\Phi}\oplus\g{n}_{\Phi})$.
As a consequence,
\[
(\g{a}\ominus\g{a}_{\Phi})\oplus(1-\theta)(\g{n}\ominus\g{n}_{\Phi})
=\Bigl(\bigoplus_{\alpha\in\Phi}\mathbb{R}H_{\alpha}\Bigr)\oplus\Bigl(\bigoplus_{\lambda\in\Sigma_{\Phi}^{+}}\g{p}_{\lambda}\Bigr)\subseteq\g{s}_{\g{p}}^\perp,
\]
where $\g{s}_\g{p}^\perp=\{\xi\in\g{p}:\langle\xi,\g{s}\rangle=0\}$, according to the definition given in Proposition~\ref{th:criterion}.
Let $\lambda\in\Sigma_{\Phi}^{+}$ be arbitrary, and $X\in\g{g}_{\lambda}$. Then, since $H_{\lambda}$,~$(1-\theta)X\in \g{s}_{\g{p}}^{\perp}$, and the action of $S$ is polar, it follows from Proposition~\ref{th:criterion} that $[H_{\lambda},(1-\theta)X]=(1+\theta)\lvert\lambda\rvert^{2}X$ is orthogonal to $\g{s}$. Thus, $\g{s}$ is orthogonal to $\bigoplus_{\lambda\in \Sigma_{\Phi}^+}\g{k}_{\lambda}$, and is therefore contained in $(\tilde{\g{t}}\cap\g{k}_{0})\oplus\g{a}_{\Phi}\oplus\g{n}_{\Phi}\subseteq (\tilde{\g{t}}\cap\g{k}_{0})\oplus\g{a}\oplus\g{n}$, with $\g{t}=\tilde{\g{t}}\cap\g{k}_{0}$ abelian.
\end{proof}

In view of Proposition~\ref{th:Borel:tan}, if $S$ is a closed solvable subgroup of $G$ acting polarly on $M$ and such that its orbits induce a foliation on $M$, we may assume from now on that the Lie algebra $\g{s}$ of $S$ is contained in a Borel subalgebra of the form $\g{t}\oplus\g{a}\oplus\g{n}$, where $\g{t}\subseteq \g{k}_{0}$ is abelian.
From now on we assume that the action of $S$ on $M$ is polar, but not hyperpolar.
Recall from Proposition~\ref{th:criterion} that $\g{s}_\g{p}^\perp$ is a Lie triple system (but not an abelian subspace) and $[\g{s}_\g{p}^{\perp},\g{s}_\g{p}^{\perp}]$ is orthogonal to~$\g{s}$.

Moreover, $[\g{s}_\g{p}^\perp, \g{s}_\g{p}^\perp]\oplus\g{s}_\g{p}^\perp$ is a reductive Lie algebra and the orbit through the origin of the subgroup $S_\g{p}^\perp$ whose Lie algebra is this one is also a symmetric space.
Since it is two-dimensional and not flat, it must be homothetic to a real hyperbolic plane $\R \mathsf{H}^2$.
Because $\R \mathsf{H}^2$ has constant curvature, it follows from~\cite[Chapter IV,~Theorem~4.2]{Helgason} that there exists a constant $C>0$ such that $\ad(\xi)^2(\eta)=C\eta$ for any pair of orthonormal vectors $\xi$, $\eta\in\g{s}_\g{p}^\perp$.
	
Our next step is to prove that $\g{s}_{\g{p}}^{\perp}$ is contained in $\g{a}\oplus\g{p}^{1}$.
We recall that $\g{n}^{1}=\bigoplus_{\alpha\in\Lambda}\g{g}_{\alpha}$ and $\g{p}^{1}=(1-\theta)\g{n}^{1}=\bigoplus_{\alpha\in\Lambda}\g{p}_{\alpha}$.
We consider the vector subspace
\[
\tilde{\g{s}}=\g{s}+(\g{n}\ominus\g{n}^{1})
=\g{s}+ \Bigl(\bigoplus_{\lambda\in \Sigma^{+}\setminus \Lambda}\g{g}_{\lambda}\Bigr).
\]

Since $\g{t}\oplus\g{a}$ normalizes all root spaces and $[\g{n},\g{n}]\subseteq \g{n}\ominus \g{n}^{1}$, it follows that $\tilde{\g{s}}$ is a subalgebra of $\g{t}\oplus\g{a}\oplus\g{n}$ containing $\g{s}$. In particular, $\g{s}_{\g{a}\oplus\g{n}}\subseteq\tilde{\g{s}}_{\g{a}\oplus\g{n}}$, so the codimension of $\tilde{\g{s}}_{\g{a}\oplus\g{n}}$ is less than or equal to two.

\begin{lemma}\label{lemma:projectionImpliesContained}
Let $\g{q}$ be a Lie subalgebra of $\g{t}\oplus\g{a}\oplus\g{n}$ and $\lambda\in \Sigma^{+}$. If $\g{g}_{\lambda}\subseteq \g{q}_{\g{a}\oplus\g{n}}$ and there exists $H\in \g{a}\cap \g{q}_{\g{a}\oplus\g{n}}$ such that $\lambda(H)\neq 0$, then $\g{g}_{\lambda}\subseteq \g{q}$.
\end{lemma}

\begin{proof}
Take $X\in\g{g}_{\lambda}\subseteq\g{q}_{\g{a}\oplus\g{n}}$. Then there exist vectors $T$, $T'\in\g{t}$ such that $T+H$, $T'+X\in\g{q}$. In particular, $\ad(T)X+\lambda(H)X=[T+H,T'+X]\in\g{q}$. This means that the linear map $\ad(T)+\lambda(H)\id_{\g{g}_{\lambda}}$ preserves $\g{g}_{\lambda}$ and carries $\g{g}_{\lambda}$ to $\g{q}$. Since $T\in\g{t}$, the linear transformation $\ad(T)$ is skew-adjoint, so $\ad(T)+\lambda(H)\id_{\g{g}_{\lambda}}$ is a linear isomorphism and it follows that $\g{g}_{\lambda}\subseteq \g{q}$.
\end{proof}

Now we can rule out the possibility that $\tilde{\g{s}}_{\g{a}\oplus\g{n}}$ has codimension zero.
	
\begin{lemma}\label{lemma:extendedProjectionIsNotAN}
$\tilde{\g{s}}_{\g{a}\oplus\g{n}}\neq\g{a}\oplus\g{n}$.
\end{lemma}

\begin{proof}
Assume that $\tilde{\g{s}}_{\g{a}\oplus\g{n}}=\g{a}\oplus\g{n}$. Both $\g{a}$ and all root spaces corresponding to positive roots are contained in $\tilde{\g{s}}_{\g{a}\oplus\g{n}}$. By Lemma~\ref{lemma:projectionImpliesContained}, it follows that $\g{n}\subseteq \tilde{\g{s}}$.

Let $m$ denote the maximum possible level of a root.
We define $k\in\{0,\dots,m\}$ to be the smallest integer for which $\g{n}^{k+1}\oplus\dots\oplus\g{n}^m\subseteq\g{s}$.
We want to show that $k=0$.
On the contrary, assume that $k\geq 1$.
Let $\lambda\in\Sigma^+$ be a root of level $k$.
As $\g{n}^1$ generates $\g{n}$, the root space $\g{g}_{\lambda}$ is generated by elements of the form
\[
\ad(X_{1})\cdots \ad(X_{k-1})X_{k}, \quad X_{i}\in \g{n}^1.
\]
Since $\g{n}\subseteq \tilde{\g{s}}=\g{s} + (\g{n}\ominus\g{n}^1)$, we can choose $Y_{1},\dots,Y_{k}\in \g{n}\ominus \g{n}^{1}$ such that $X_{i}+Y_{i}\in \g{s}$ for each $i\in\{1,\dots,k\}$.
Hence,
$\ad(X_{1}+Y_{1})\cdots\ad(X_{k-1}+Y_{k-1})(X_{k}+Y_{k})\in\g{s}$.
By using the fact that $[\g{n}^r,\g{n}^s]\subseteq\g{n}^{r+s}$, we have
\begin{multline*}
\ad(X_{1}+Y_{1})\cdots\ad(X_{k-1}+Y_{k-1})(X_{k}+Y_{k})\equiv\\
\ad(X_{1})\cdots \ad(X_{k-1})X_{k} 
\pmod{\g{n}^{k+1}\oplus\dots\oplus\g{n}^m},
\end{multline*}
so we obtain $\ad(X_{1})\cdots\ad(X_{k-1})X_{k}\in \g{s}$. 
This means that $\g{g}_{\lambda}\subseteq \g{s}$, and as a result, $\g{n}^{k}\subseteq \g{s}$, contradicting the definition of $k$.

Therefore, $k=0$ and $\g{n}\subseteq\g{s}$.
In particular, $\g{s}_{\g{p}}^{\perp}\subseteq \g{a}$ must be an abelian subspace, contradicting the fact that our action is not hyperpolar. Thus, the case $\tilde{\g{s}}_{\g{a}\oplus\g{n}}=\g{a}\oplus\g{n}$ is not possible.
\end{proof}

Before analyzing the remaining possibilities for the codimension of $\tilde{\g{s}}_{\g{a}\oplus\g{n}}$ we need the following result.

\begin{lemma}\label{lemma:normalap1}
Assume that $V\in\g{a}\oplus\g{n}$ is nonzero and orthogonal to $\g{s}$. Then:
\begin{enumerate}[\rm (i)]
\item If $V\in \g{a}$, then $\g{s}_\g{p}^{\perp}=\R V\oplus(1-\theta)\R\eta_\alpha$, where $\eta_{\alpha}\in \g{g}_{\alpha}$ is nonzero and $\alpha\in\Lambda$.
Furthermore, $V$ is proportional to $H_{\alpha}$. \label{lemma:normalap1:a}
\item If $V\in \g{g}_{\alpha}$ for $\alpha\in \Lambda$, then $\g{s}_\g{p}^{\perp}=(1-\theta)\bigl(\R V\oplus\R(aH_{\alpha}+\eta_{\alpha})\bigr)$, where $a\in\mathbb{R}$, $\eta_{\alpha}\in\g{g}_{\alpha}\ominus\R V$, and $[V,\eta_\alpha]=0$.\label{lemma:normalap1:ga}
\item If $V=H_{\alpha}+\xi_{\alpha}$, where $\alpha \in \Lambda$ and $\xi_{\alpha}\in \g{g}_\alpha$ is a nonzero vector, and $g=\Exp(-\xi_{\alpha}/\lvert\xi_{\alpha}\rvert^{2})\in N$, then $\Ad(g)\g{s}\subseteq \g{t}\oplus\g{a}\oplus\g{n}$ is orthogonal to $\xi_{\alpha}$.
In particular, $(\Ad(g)\g{s})_\g{p}^{\perp}=(1-\theta)\bigl(\R\xi_{\alpha}\oplus\R(a H_{\alpha} + \eta_{\alpha})\bigr)$, for $a\in \mathbb{R}$ and $\eta_{\alpha}\in\g{g}_{\alpha}\ominus\R\xi_{\alpha}$. Furthermore, $[\xi_\alpha,\eta_\alpha]=0$.\label{lemma:normalap1:aga}
\end{enumerate}
\end{lemma}

\begin{proof}
We prove~(\ref{lemma:normalap1:a}).
Assume $V \in \g{a}$.
Since $\g{a}\subseteq \g{p}$, this means $V\in \g{s}_{\g{p}}^{\perp}$.
Choose any unit vector $\eta=\eta_{0}+\sum_{\lambda\in\Sigma^{+}}(1-\theta)\eta_{\lambda}\in \g{s}_{\g{p}}^{\perp}$ orthogonal to $V$, where $\eta_{0}\in \g{a}$ and $\eta_{\lambda}\in \g{g}_{\lambda}$ for each $\lambda\in \Sigma^{+}$.
Since the action of $S$ is polar nonhyperpolar, $[V,\eta]$ is a nonzero vector orthogonal to $\g{s}$. Note that
\[
[V,\eta]=(1+\theta)\sum_{\lambda\in\Sigma^{+}}\lambda(V)\eta_{\lambda},
\]
and recalling that $\theta\g{n}$ is orthogonal to $\g{s}$, we obtain
\[
(1-\theta)\sum_{\lambda\in\Sigma^{+}}\lambda(V)\eta_{\lambda}\in \g{s}_{\g{p}}^{\perp}\ominus \R V=\R \eta,
\]
so $\eta_{0}=0$ and $\lambda(V)=\mu(V)$ for every pair of roots $\lambda$, $\mu\in\Sigma^{+}$ such that $\eta_{\lambda},\eta_{\mu}\neq 0$.

Suppose $\mu$, $\nu\in\Sigma^{+}$ are two different roots such that $\eta_{\mu},\eta_{\nu}\neq 0$.
Hence, $\langle H_{\mu-\nu},V \rangle=\mu(V)-\nu(V)=0$ and $\langle H_{\mu-\nu},\eta \rangle=0$, so $H_{\mu-\nu}\in\g{s}_{\g{a}\oplus\g{n}}$.
Therefore, we may choose a $T\in \g{t}$ such that $T+H_{\mu-\nu}\in \g{s}$.
An analogous argument shows that $\eta_{\mu}+a\eta_{\nu}\in \g{s}_{\g{a}\oplus\g{n}}$ for $a=-\lvert\eta_{\mu}\rvert^{2}/\lvert\eta_{\nu}\rvert^{2}<0$, so there is a vector $T'\in \g{t}$ such that $T'+\eta_{\mu}+a\eta_{\nu}\in \g{s}$.
Hence,
\[
[T,\eta_{\mu}]+a[T,\eta_{\nu}]+\langle\mu-\nu,\mu\rangle\eta_{\mu}+a\langle\mu-\nu,\nu\rangle\eta_{\nu}
=[T+H_{\mu-\nu},T'+\eta_{\mu}+a\eta_{\nu}]\in\g{s}\cap\g{n}.
\]
Since $\ad(T)$ is skew-adjoint (because $T\in \g{k}$) we deduce $[T,\eta_{\mu}]$, $[T,\eta_{\nu}]\in\g{s}_{\g{a}\oplus\g{n}}$. 
Thus $\langle\mu-\nu,\mu\rangle\eta_{\mu}+a\langle\mu-\nu,\nu\rangle\eta_{\nu}$ is also in $\g{s}_{\g{a}\oplus\g{n}}$. Observe that
\[
\begin{vmatrix}
1 & a \\
\langle\mu-\nu,\mu\rangle & a\langle\mu-\nu,\nu\rangle
\end{vmatrix}
=-a\lvert\mu-\nu\rvert^{2}>0,
\]
which implies that $\eta_{\mu}+a\eta_{\nu}$ and $\langle\mu-\nu,\mu\rangle\eta_{\mu}+a\langle\mu-\nu,\nu\rangle\eta_{\nu}$ are linearly independent vectors in $\g{s}_{\g{a}\oplus\g{n}}$.
Therefore, $\eta_{\mu}$, $\eta_{\nu}\in\g{s}_{\g{a}\oplus\g{n}}$ must be orthogonal to $\eta$, contradicting the fact that they are nonzero.
We thus obtain that only one of the $\eta_{\lambda}$ can be nonzero, that is, $\eta=(1-\theta)\eta_{\mu}\in \g{p}_{\mu}$ for some $\mu\in \Sigma^{+}$.

We now prove that $\mu$ is simple.
If $\mu=\beta+\gamma$ were a sum of positive roots $\beta$, $\gamma\in \Sigma^{+}$, then $\eta_{\mu}\in \g{g}_{\mu}=[\g{g}_{\beta},\g{g}_{\gamma}]$, so we may write $\eta_{\nu}=\sum_{i=1}^{k}[X_{i},Y_{i}]$, where $X_{i}\in \g{g}_{\beta}$ and $Y_{i}\in \g{g}_{\gamma}$ for each $i$.
Clearly, $\g{g}_{\beta}+\g{g}_{\gamma}\subseteq \g{s}_{\g{a}\oplus\g{n}}$, which means that for each $i$ there are vectors $T_{i}$, $T_{i}'\in \g{t}$ such that $T_{i}+X_{i}$, $T_{i}'+Y_{i}\in \g{s}$.
As a consequence,
\[
\sum_{i=1}^{k}[T_{i},Y_{i}]+[X_{i},T_{i}']+[X_{i},Y_{i}]
=\sum_{i=1}^{k}[T_{i}+X_{i},T_{i}'+Y_{i}]\in\g{s}.
\]
Note that each $[T_{i},Y_{i}]$ is in $\g{g}_{\gamma}\subseteq \g{s}_{\g{a}\oplus\g{n}}$ and each $[X_{i},T_{i}']$ is in $\g{g}_{\beta}\subseteq \g{s}_{\g{a}\oplus\g{n}}$, which implies that $\eta_{\mu}=\sum_{i}[X_{i},Y_{i}]\in \g{s}_{\g{a}\oplus\g{n}}$, contradicting that $\eta_{\mu}$ is nonzero. We deduce that $\mu\in \Lambda$, so we may write $\g{s}_{\g{p}}^{\perp}=\R V\oplus (1-\theta)\eta_{\alpha}$, where $\alpha=\mu\in \Lambda$.

As $V$ and $(1-\theta)\eta_{\alpha}$ are orthogonal, we know that the vector $\ad((1-\theta)\eta_{\alpha})^{2}V$ is nonzero and proportional to $V$.
Taking the inner product with an arbitrary vector $H\in \g{a}$, we see that
\[
	\begin{aligned}
		\langle \ad((1-\theta)\eta_{\alpha})^{2}V,H \rangle ={}&\langle [(1-\theta)\eta_{\alpha},V],[(1-\theta)\eta_{\alpha},H] \rangle \\
		={}& \langle (1+\theta)\alpha(V)\eta_{\alpha},(1+\theta)\alpha(H)\eta_{\alpha} \rangle=2\lvert \eta_{\alpha}\rvert^{2}\alpha(V)\alpha(H) \\
		={}&\langle 2 \lvert \eta_{\alpha}\rvert^{2}\alpha(V)H_{\alpha},H \rangle,
	\end{aligned}
\]
which means that $V$ is proportional to $\ad((1-\theta)\eta_{\alpha})^{2}V= 2 \lvert \eta_{\alpha}\rvert^{2}\alpha(V)H_{\alpha}$ (this vector is nonzero due to the action not being hyperpolar), and thus to $H_{\alpha}$.
This proves the first assertion.

Now we prove~(\ref{lemma:normalap1:ga}).
Assume $V \in \g{g}_{\alpha}$ for a simple root $\alpha\in\Lambda$.
Choose a nonzero vector $\eta=\eta_{0}+\sum_{\lambda\in\Sigma^{+}}(1-\theta)\eta_{\lambda}\in\g{s}_{\g{p}}^{\perp}$, where $\eta_{0}\in\g{a}$ and $\eta_{\lambda}\in \g{g}_{\lambda}$ for each $\lambda\in \Sigma^{+}$, such that $\langle V,\eta \rangle=\langle V,\eta_{\alpha} \rangle=0$.
We may assume that $\eta$ is not in $\g{a}$.
Indeed, if $\eta\in\g{a}$, then the we are in the conditions of item~(\ref{lemma:normalap1:a}), and we directly obtain that $\eta$ is proportional to $H_{\alpha}$, proving our claim.
Let $\mu\in \Sigma^{+}$ be a positive root with $\eta_{\mu}\neq 0$.

For now, let us suppose that we can choose $\mu$ so that $\mu\neq \alpha$.

We first prove that $\eta_{0}\in \R H_{\mu}$.
Assume otherwise, so there exists a vector $H\in\g{a}$ such that $\langle H,\eta_{0} \rangle=0$ and $\langle H,H_{\mu}\rangle=\mu(H)\neq 0$.
Then $H\in \g{s}_{\g{a}\oplus\g{n}}$, so there exists $T\in \g{t}$ such that $T+H\in \g{s}$.
On the other hand, $\eta_{0}+a\eta_{\mu}$ is orthogonal to both $V$ and $\eta$ for $a=-\lvert\eta_{0}\rvert^{2}/\lvert\eta_{\mu}\rvert^{2}<0$, which implies that $\eta_{0}+a\eta_{\mu}\in \g{s}_{\g{a}\oplus\g{n}}$, and there exists $T'\in \g{t}$ such that $T'+\eta_{0}+a \eta_{\mu}\in \g{s}$.
In particular, $a[T,\eta_{\mu}]+a\mu(H)\eta_{\mu}=[T+H,T'+\eta_{0}+a\eta_{\mu}]\in\g{s}$, 
so $0=\langle a[T,\eta_{\mu}]+a\mu(H)\eta_{\mu},\eta\rangle=a\mu(H)\lvert\eta_{\mu}\rvert^{2}$, a contradiction.

We now prove that $\eta_{\lambda}=0$ for every $\lambda\in\Sigma^{+}\setminus\{\mu\}$.
If $\lambda$ is a positive root linearly independent with $\mu$, and $\eta_{\lambda}\neq 0$, we may find $H\in \g{a}$ such that $\mu(H)=0\neq \lambda(H)$. Since $\eta_{0}\in \R H_{\mu}$, this implies that $H\in \g{s}_{\g{a}\oplus\g{n}}$, and hence, there exists $T\in \g{t}$ such that $T+H\in \g{s}$.
Furthermore, $\eta_{\mu}+b\eta_{\lambda}\in\g{s}_{\g{a}\oplus\g{n}}$ with $b=-\lvert\eta_{\mu}\rvert^{2}/\lvert\eta_{\lambda}\rvert^{2}<0$, and there exists $T'\in \g{t}$ satisfying $T'+\eta_{\mu}+b\eta_{\lambda}\in\g{s}$.
As a consequence, $[T,\eta_{\mu}]+b[T,\eta_{\lambda}]+b\lambda(H)\eta_{\lambda}=[T+H,T'+\eta_{\mu}+b \eta_{\lambda}]\in \g{s}$.
In particular, $0=\langle [T,\eta_{\mu}]+b[T,\eta_{\lambda}]+b\lambda(H)\eta_{\lambda},\eta \rangle=b\lambda(H)\lvert\eta_{\lambda}\rvert^{2}$, a contradiction.
On the other hand, if $2\mu\in \Sigma^{+}$ and $\eta_{2\mu}\neq 0$, a similar argument yields that there is $T\in \g{t}$ such that $T+H_{\mu}+a \eta_{\mu}\in\g{s}$, for $a=-\mu(\eta_{0})/\lvert\eta_{\mu}\rvert^{2}$, and $T'\in \g{t}$ such that $T'+\eta_{\mu}+b\eta_{2\mu}\in \g{s}$, where $b=-\lvert\eta_{\mu}\rvert^{2}/\lvert\eta_{2\mu}\rvert^{2}<0$.
Thus,
\begin{align*}
0&{}=\langle [T+H_{\mu}+a \eta_{\mu},T'+\eta_{\mu}+b\eta_{2\mu}],\eta\rangle\\
&{}=
\langle [T,\eta_{\mu}]+b[T,\eta_{2\mu}]+\lvert\mu\rvert^{2}\eta_{\mu}+2b\lvert\mu\rvert^{2}\eta_{2\mu}+a[\eta_{\mu},T'], \eta\rangle=-\lvert\mu\rvert^{2}\lvert\eta_{\mu}\rvert^{2},
\end{align*}
contradicting our choice of $\mu$. This implies $\eta_{2\mu}=0$ (and an analogous argument shows that $\eta_{\mu/2}=0$ in the case that $\mu/2 \in \Sigma^{+}$).

To summarize, we have obtained $\eta=a H_{\mu}+(1-\theta)\eta_{\mu}$ for some constant $a\in \R$, and $\mu$ is a root different from $\alpha$ with $\eta_{\mu}\neq 0$.

Assume $a\neq 0$. Since $\g{s}_\g{p}^\perp$ is a Lie triple system, $\ad((1-\theta)V)^2\eta$ is proportional to $\eta$.
Since for any $H\in\g{a}$ we have
\begin{align*}
\langle\ad((1-\theta)V)^2\eta,H\rangle
&{}=-\langle[(1-\theta)V,\eta],[H,(1-\theta)V]\rangle\\
&{}=-\langle(1+\theta)\bigl( -a\langle \alpha,\mu\rangle V + [V,\eta_{\mu}]-[\theta V, \eta_{\mu}] \bigr),\alpha(H)(1+\theta)V\rangle\\
&{}=2a\langle\alpha,\mu\rangle\lvert V\rvert^{2}\alpha(H)
=\langle 2a\langle\alpha,\mu\rangle \lvert V\rvert^{2} H_\alpha, H\rangle,
\end{align*}
it follows that $\mu=\alpha$ or $\mu=2\alpha$.
This last case is not possible, because by Proposition~\ref{th:criterion}, $(1+\theta)(-2a\lvert\alpha\rvert^{2}V-[\theta V,\eta_{2\alpha}])=[(1-\theta)V,aH_{2 \alpha}+(1-\theta)\eta_{2\alpha}]$ would be orthogonal to $\g{s}$.
This would imply that $[\theta V,\eta_{2\alpha}]$ is proportional to $V$, and thus,
\[
0=[V,[\theta V,\eta_{2 \alpha}]]=-[\theta V,[\eta_{2 \alpha},V]]-[\eta_{2 \alpha},[V,\theta V]]
=-[\lvert V\rvert^{2}H_{\alpha},\eta_{2 \alpha}]=-2\lvert\alpha\rvert^{2}\lvert V\rvert^{2}\eta_{2 \alpha},
\]
contradicting the fact that $\eta_{2\alpha}\neq 0$. We conclude that $\mu=\alpha$.

Suppose now that $a=0$, so $\eta\in \g{p}_{\mu}$.
Since $\g{a},\g{g}_{\alpha+\mu}\subseteq\g{s}_{\g{a}\oplus\g{n}}$, we obtain $\g{g}_{\alpha+\mu}\subseteq \g{s}$ by Lemma~\ref{lemma:projectionImpliesContained}.
The vector $[(1-\theta)V,(1-\theta)\eta_{\mu}]=(1+\theta)([V,\eta_{\mu}]-[\theta V,\eta_{\mu}])$ is nonzero and orthogonal to $\g{s}$. Combining this with the fact that $\g{g}_{\alpha+\mu}\subseteq \g{s}$, we get $[V,\eta_{\mu}]=0$, so $(1+\theta)[\theta V,\eta_{\mu}]$ is nonzero and orthogonal to $\g{s}$.
This now implies $\mu-\alpha\in\Sigma^+$ and $[\theta V,\eta_{\mu}]\in (\g{a}\oplus\g{n})\ominus\g{s}$.
Furthermore, we must have $\mu=2\alpha$, and thus,
$[\theta V,\eta_{\mu}]\in ( (\g{a}\oplus\g{n})\ominus\g{s})\cap \g{g}_{\alpha}=\mathbb{R}V$. Therefore,
\[
0=[V,[\theta V,\eta_{2\alpha}]]=-[\eta_{2\alpha},[V,\theta V]]=-2\lvert \alpha\rvert^{2}\lvert V\rvert^{2}\eta_{2\alpha},
\]
which yields a contradiction.

We now assume $\eta_{\mu}= 0$ for every $\mu\in\Sigma^{+}\setminus \{\alpha\}$.
As a consequence, $\eta=\eta_{0}+(1-\theta)\eta_{\alpha}$, with $\langle V, \eta_{\alpha}\rangle = 0$.
We only need to prove that $\eta_{0}\in\mathbb{R}H_{\alpha}$.
Indeed, if $\eta_{0}$ is not proportional to $H_{\alpha}$, there exists $H\in\g{a}$ such that $\langle H, \eta_{0}\rangle = 0$ and $\alpha(H)\neq 0$.
Therefore, $H\in \g{s}_{\g{a}\oplus\g{n}}$, so we may find $T\in \g{t}$ satisfying $T+H\in\g{s}$.
Similarly, by taking $\eta_{0}+x\eta_{\alpha}$ with $x=-\lvert\eta_{0}\rvert^{2}/\lvert\eta_{\alpha}\rvert^{2}<0$, we obtain $\eta_{0}+x\eta_{\alpha}\in \g{s}_{\g{a}\oplus\g{n}}$, so $T'+\eta_{0}+ x \eta_{\alpha}\in \g{s}$ for an adequate $T'\in \g{t}$.
Thus,
$0=\langle[T+H,T'+\eta_{0}+ x \eta_{\alpha}],\eta\rangle=x\alpha(H)\lvert\eta_\alpha\rvert^2$, contradiction.
Hence, we may write $\eta=a H_\alpha+(1-\theta)\eta_\alpha$ with $a\in\R$, $\eta_\alpha\in\g{g}_\alpha\ominus\R V$, and $\g{s}_\g{p}^\perp=\R(1-\theta)V\oplus\R\eta$.

If $a=0$, then note that $\g{a}$, $\g{g}_{2\alpha}\subseteq\g{s}_{\g{a}\oplus\g{n}}$, and Lemma~\ref{lemma:projectionImpliesContained} implies $\g{g}_{2\alpha}\subseteq \g{s}$.
Together with the fact that $[(1-\theta)V,(1-\theta)\eta_\alpha]=(1+\theta)([V,\eta_\alpha]-[\theta V,\eta_\alpha])$ is orthogonal to $\g{s}$ by Proposition~\ref{th:criterion}, we get $[V,\eta_{\alpha}]=0$.

If $a\neq0$, we can take the triple bracket $[\eta,[\eta,(1-\theta)V]]$, which is in $\g{s}_{\g{p}}^{\perp}\subseteq\g{a}\oplus\g{p}^{1}$.
Then, for any $X\in\g{g}_{2\alpha}$:
\begin{align*}
0
&{}=\langle [\eta,[\eta,(1-\theta) V]],X \rangle
=\langle [\eta,(1-\theta)V], [\eta,X]] \rangle \\
&{}=\langle(1+\theta)(a\lvert\alpha\rvert^2 V+[\eta_\alpha,V]-[\theta\eta_\alpha,V]),2a\lvert\alpha\rvert^2 X-[\theta\eta_\alpha,X]\rangle\\
&{}=-a\lvert\alpha\rvert^2\langle V,[\theta\eta_\alpha,X]\rangle+2a\lvert\alpha\rvert^2\langle[\eta_\alpha,V],X\rangle
=-3a\lvert\alpha\rvert^2\langle[V,\eta_\alpha],X\rangle,
\end{align*}
and this yields $[V,\eta_{\alpha}]=0$, as stated.
This finishes the proof of~(\ref{lemma:normalap1:ga}).

To prove~(\ref{lemma:normalap1:aga}), assume $V=H_{\alpha}+\xi_{\alpha}$ for a nonzero $\xi_{\alpha}\in \g{g}_{\alpha}$, where $\alpha\in\Lambda$, and consider $g=\Exp(-\xi_{\alpha}/\lvert\xi_{\alpha}\rvert^{2})\in N$.
Then, the isomorphism $\Ad(g)$ preserves the subalgebra $\g{t}\oplus\g{a}\oplus\g{n}$, and therefore, $\Ad(g)\g{s}\subseteq \g{t}\oplus \g{a}\oplus \g{n}$ is the Lie algebra of $gSg^{-1}$, which induces a homogeneous polar foliation on $M$ with a non flat section. Observe that
\[
\Ad(g^{-1})^{*}(H_{\alpha}+\xi_{\alpha})=
e^{-\frac{1}{\lvert\xi_{\alpha}\rvert^{2}}\ad(\theta\xi_{\alpha})}(H_{\alpha}+\xi_{\alpha})
=\xi_{\alpha}-\frac{\lvert\alpha\rvert^{2}}{2\lvert\xi_{\alpha}\rvert^{2}}\theta\xi_{\alpha},
\]
which means that $\Ad(g)\g{s}\subseteq \g{t}\oplus\g{a}\oplus\g{n}$ is orthogonal to $\xi_{\alpha}$, as desired. The rest of the assertion follows from~(\ref{lemma:normalap1:ga}).
\end{proof}

Now we continue with the proof of Theorem~\ref{th:MainTheoremList}.
Recall that $\tilde{\g{s}}=\g{s}+(\g{n}\ominus\g{n}^1)$.
According to Lemma~\ref{lemma:extendedProjectionIsNotAN}, $\tilde{\g{s}}_{\g{a}\oplus\g{n}}\neq\g{a}\oplus\g{n}$, so either $\tilde{\g{s}}_{\g{a}\oplus\g{n}}$ has codimension one or two in $\g{a}\oplus\g{n}$.

If $\tilde{\g{s}}_{\g{a}\oplus\g{n}}$ has codimension two, we have $\tilde{\g{s}}_{\g{a}\oplus\g{n}}=\g{s}_{\g{a}\oplus\g{n}}$, which is equivalent to $\g{s}_\g{p}^\perp\subseteq\g{a}\oplus\g{p}^1$.

On the other hand, if $\tilde{\g{s}}_{\g{a}\oplus\g{n}}$ is a codimension one subspace of $\g{a}\oplus\g{n}$, then by~\cite[Proposition 5.4]{BerndtTamaruC1Foliations} we have $\tilde{\g{s}}_{\g{a}\oplus\g{n}}=(\g{a}\oplus\g{n})\ominus\mathbb{R}\xi$, where $\xi$ satisfies one of the following possibilities:\footnote{Note that Berndt and Tamaru's proof does not rely on the additional condition that they impose on $\g{s}$, namely, that $\g{s}\cap\g{t}=0$.}
\begin{enumerate}[\rm (i)]
\item $\xi\in\g{a}$.
\item $\xi\in \g{g}_{\alpha}$ for a simple root $\alpha\in\Lambda$.
\item $\xi = H_{\alpha}+\xi_{\alpha}$, where $\xi_{\alpha}\in \g{g}_{\alpha}$ is a nonzero vector and $\alpha\in\Lambda$.
\end{enumerate}
Note that $\xi$ is also orthogonal to $\g{s}$.
Hence, by Lemma~\ref{lemma:normalap1}, we obtain that $\g{s}_{\g{p}}^{\perp}$ may be assumed to be in $\g{a}\oplus\g{p}^{1}$ after conjugation by an element of $N$. 
The next step is to determine the orthogonal projection $\g{s}_{\g{a}\oplus\g{n}}$.
	
\begin{lemma}\label{lemma:normalvectorinp1}
The action of $S$ is orbit equivalent to the action of a connected closed subgroup $\bar{S}$ whose Lie algebra $\bar{\g{s}}$ is contained in $\g{t}\oplus\g{a}\oplus\g{n}$, the normal space $\bar{\g{s}}_{\g{p}}^{\perp}$ is contained in $\g{a}\oplus\g{p}^{1}$, and $\bar{\g{s}}_{\g{p}}^{\perp}\cap \g{p}^{1}\neq 0$. Equivalently, the orthogonal projection of $\bar{\g{s}}_{\g{p}}^{\perp}$ on $\g{a}$ is not two-dimensional.
\end{lemma}
	
\begin{proof}
Assume $\g{s}_{\g{p}}^{\perp}\cap \g{p}^{1}$ is trivial. 
Let $\Psi=\{ \alpha\in \Lambda\colon \pi_{\g{g}_{\alpha}}(\g{s}_\g{p}^{\perp})\neq 0 \}$, where $\pi_{\g{g}_\alpha}\colon\g{g}\to\g{g}_\alpha$ denotes the orthogonal projection.
Since the action is not hyperpolar, $\Psi$ is a nonempty subset of $\Lambda$. 
We prove $\g{a}^{\Psi}=\bigoplus_{\alpha\in \Psi}\mathbb{R}H_{\alpha}\subseteq \pi_{\g{a}}(\g{s}_\g{p}^{\perp})$. 
Here $\pi_\g{a}$ denotes the orthogonal projection onto $\g{a}$.
Indeed, assume $H\in \g{a}\ominus \pi_{\g{a}}(\g{s}_\g{p}^{\perp})$. Then there exists a vector $T\in \g{t}$ such that $T+H\in \g{s}$.
On the other hand, let $\alpha\in \Psi$. 
We may find two vectors $\xi=\xi_{0}+\sum_{\beta\in\Lambda}(1-\theta)\xi_{\beta}$ and $\eta=\eta_{0}+\sum_{\beta\in\Lambda}(1-\theta)\eta_{\beta}$ in $\g{s}_\g{p}^{\perp}$ such that $\xi_{\alpha}\neq 0$ and $\langle\xi_{\alpha},\eta_{\alpha}\rangle=0$. 
By our assumption, $\xi_{0}$ and $\eta_{0}$ are linearly independent, which implies that there exist unique constants $x$, $y\in \mathbb{R}$ such that $\xi_{\alpha}+x\xi_{0}+y\eta_{0}\in \g{s}_{\g{a}\oplus\g{n}}$. 
In particular, we may find $T'\in \g{t}$ satisfying $T'+\xi_{\alpha}+x\xi_{0}+y\eta_{0}\in \g{s}$. 
Thus, $[T,\xi_{\alpha}]+\alpha(H)\xi_{\alpha}=[T+H,T'+\xi_{\alpha}+x\xi_{0}+y\eta_{0}]\in \g{s}$. 
Taking the inner product with $\xi$, we deduce $\alpha(H)=0$. 
All in all, we obtain $\g{a}\ominus\pi_{\g{a}}(\g{s}_\g{p}^{\perp})\subseteq\g{a}_{\Psi}=\cap_{\alpha\in\Psi}\ker \alpha$, so $\g{a}^{\Psi}\subseteq \pi_{\g{a}}(\g{s}^{\perp})$ (in particular, $\Psi$ has either one or two elements).

In order to prove the result, we assume first that for all $\alpha\in \Psi$, the orthogonal projection $\pi_{\g{g}_{\alpha}}(\g{s}_\g{p}^{\perp})$ is one-dimensional. 
Thus, we can take two orthogonal vectors $\xi=\xi_{0}+\sum_{\alpha\in\Psi}(1-\theta)\xi_{\alpha}$ and $\eta=\eta_{0}+\sum_{\alpha\in\Psi}(1-\theta)\eta_{\alpha}$ that span $\g{s}_{\g{p}}^{\perp}$. 
Since the action is polar nonhyperpolar, the vector $[\xi,\eta]$ is nonzero and orthogonal to $\g{s}$. 
Observe that, because $\Psi\subseteq \Lambda$, $[\theta \xi_{\alpha},\eta_{\alpha}]\in \g{a}$ for all $\alpha\in \Psi$, and $\alpha-\beta\notin\Sigma$, we have
\[
[\xi,\eta]
=(1+\theta)\left(\sum_{\alpha\in\Psi}
\Bigl(\alpha(\xi_{0})\eta_{\alpha}-\alpha(\eta_{0})\xi_{\alpha}\Bigr)
+\sum_{\alpha,\beta\in\Psi}[\xi_{\alpha},\eta_{\beta}]\right).
\]
Since $\theta\g{n}$ and $\g{s}$ are orthogonal, we obtain 
\[
(1-\theta)\left(\sum_{\alpha\in\Psi}
\Big(\alpha(\xi_{0})\eta_{\alpha}-\alpha(\eta_{0})\xi_{\alpha}\Bigr)
+\sum_{\alpha,\beta\in\Psi}[\xi_{\alpha},\eta_{\beta}]\right)\in \g{s}_{\g{p}}^{\perp}\subseteq \g{a}\oplus\g{p}^{1},
\]
which means that all terms in $\g{p}^{2}$ cancel out and $\g{s}_{\g{p}}^{\perp}\cap\g{p}^{1}\neq 0$, a contradiction.

Now, assume that there exists $\alpha\in \Psi$ such that $\pi_{\g{g}_{\alpha}}(\g{s}_\g{p}^{\perp})$ is two-dimensional. 
Since $H_{\alpha}\in \pi_{\g{a}}(\g{s}_\g{p}^{\perp})$, we may find $H_{\alpha}+\sum_{\beta\in\Psi}(1-\theta)\xi_{\beta}\in\g{s}_\g{p}^{\perp}$, 
with each $\xi_{\beta}\in\g{g}_{\beta}$, and $\xi_{\alpha}\neq 0$, for dimension reasons. 
Consider the element $g=\Exp(-\xi_{\alpha}/\lvert\xi_{\alpha}\rvert^{2})\in N$. 
Then the action of $S$ is orbit equivalent to the action of $gSg^{-1}$, whose Lie algebra is $\Ad(g)\g{s}\subseteq \g{t}\oplus\g{a}\oplus\g{n}$. 
Note that the equality
\[
\Ad(g^{-1})^{*}\biggl(H_{\alpha}+\sum_{\beta_\in\Psi}\xi_{\beta}\biggr)
=\sum_{\beta\in\Psi}\xi_{\beta}
-\frac{\lvert\alpha\rvert^{2}}{2\lvert\xi_{\alpha}\rvert^{2}}\theta\xi_{\alpha}
\]
and $\langle\theta\g{n},\g{s}\rangle=0$ imply $\sum_{\beta\in\Psi}(1-\theta)\xi_{\beta}
\in(\Ad(g)\g{s})_\g{p}^{\perp}\cap\g{p}^{1}$.

To conclude, it suffices to prove that 
$(\Ad(g)\g{s})_{\g{p}}^{\perp}\subseteq \g{a}\oplus\g{p}^{1}$. 
This is the case if the projection of $\Ad(g)\g{s}+(\g{n}\ominus\g{n}^1)$ onto $\g{a}\oplus\g{n}$ has codimension~2.
As the projection of $\Ad(g)\g{s}+(\g{n}\ominus\g{n}^1)$ onto $\g{a}\oplus\g{n}$ cannot be $\g{a}\oplus\g{n}$ by Lemma~\ref{lemma:extendedProjectionIsNotAN}, our assertion is false whenever this projection is of codimension one, that is, when the orthogonal complement of $\Ad(g)\g{s}$ in $\g{a}\oplus\g{n}$ is spanned by
$\sum_{\beta\in\Psi}\xi_{\beta}$. 
By~\cite[Proposition 5.4]{BerndtTamaruC1Foliations}, $\xi_{\beta}= 0$ for all simple roots $\beta\neq \alpha$, and by Lemma~\ref{lemma:normalap1}(\ref{lemma:normalap1:ga}) we have $(\Ad(g)\g{s})_\g{p}^{\perp}=(1-\theta)\bigl(\R\xi_{\alpha}\oplus(a H_{\alpha}+\eta_{\alpha})\bigr)$ for $a\in \mathbb{R}$ and $\eta_{\alpha}\in \g{g}_{\alpha}$. Thus, $(\Ad(g)\g{s})_{\g{p}}^{\perp}
\subseteq\g{a}\oplus\g{p}^{1}$, contradicting the fact that the projection of $\Ad(g)\g{s}+(\g{n}\ominus\g{n}^1)$ onto $\g{a}\oplus\g{n}$ has codimension one.
\end{proof}

Due to the previous lemma, we may assume that $\g{s}_{\g{p}}^{\perp}\cap \g{p}^{1}$ is a nonzero subspace of $\g{g}$. 

First we need:

\begin{lemma}\label{lemma:isotropyAlgebras}
Let $S$ be a closed subgroup of $G$ whose Lie algebra $\g{s}$ satisfies $\g{s}\subseteq \g{t}\oplus\g{a}\oplus\g{n}$, and the orbits of $S$ form a homogeneous foliation on $M$. 
Let $V\in\g{n}$ be a vector such that $(1-\theta)V\in\g{s}_{\g{p}}^{\perp}$ and $g=\Exp(V)\in N$. Then $\g{s}\cap\g{t}=\Ad(g)(\g{s}\cap \g{t})=\Ad(g)(\g{s})\cap\g{k}$.
\end{lemma}

\begin{proof}
Since the orbit $S\cdot o$ is principal, we have $[\g{s}\cap\g{t},\g{s}_{\g{p}}^{\perp}]=0$. 
Hence, $[\g{s}\cap\g{t},(1-\theta)V]=(1-\theta)[\g{s}\cap\g{t},V]=0$, which means $\ad(V)(\g{s}\cap \g{t})=0$ and $\Ad(g)(\g{s}\cap\g{t})=e^{\ad(V)}(\g{s}\cap\g{t})=\g{s}\cap\g{t}$. 
On the other hand, $\g{s}\cap\Ad(g^{-1})\g{k}$ is the isotropy algebra of $S$ at $g^{-1}\cdot o$, and since all orbits have the same type, it follows that $\dim \Ad(g)\g{s}\cap \g{k}=\dim \g{s}\cap\Ad(g^{-1})\g{k}=\dim\g{s}\cap\g{k}=\dim\g{s}\cap\g{t}$. 
Since $\g{s}\cap\g{t}=\Ad(g)(\g{s}\cap\g{t})\subseteq \Ad(g)(\g{s})\cap\g{k}$, the equality follows.
\end{proof}

The next result is needed later to handle the two examples in Theorem~\ref{th:MainTheoremList} simultaneously.

\begin{proposition}\label{prop:orbit-equivalence}
Let $S$ be a connected closed subgroup of $G$ inducing a homogeneous polar foliation on $M$.
Assume that its Lie algebra is contained in a maximally noncompact Borel subalgebra $\g{t}\oplus\g{a}\oplus\g{n}$,
and $\g{s}_{\g{a}\oplus\g{n}}=\g{z}\oplus(\g{n}\ominus\g{v}_\alpha)$, where $\g{v}_\alpha$ is an abelian subspace of $\g{g}_\alpha$, $\alpha\in\Lambda$, and $\g{z}$ is a subspace of $\g{a}$.
Let $\tilde{S}$ be the connected Lie subgroup of $G$ whose Lie algebra is $\tilde{\g{s}}=\g{z}\oplus(\g{n}\ominus\g{v}_\alpha)$.
Then, $S$ and $\tilde{S}$ have the same orbits.
\end{proposition}

\begin{proof}
Denote by $\g{s}_{\g{t}}$ the orthogonal projection of $\g{s}$ onto $\g{t}$. 
We start by proving that $\hat{\g{s}}=\g{s}_{\g{t}}\oplus\tilde{\g{s}}$ is a Lie subalgebra of $\g{g}$.
Since $[\g{g}_\lambda,\g{g}_\mu]\subseteq\g{g}_{\lambda+\mu}$, for $\lambda$, $\mu\in\Sigma^+$, and $\g{t}$ is abelian, centralizes $\g{a}$, and normalizes each root space, this amounts to proving the inclusion $[\g{s}_{\g{t}},\g{g}_\alpha\ominus\g{v}_\alpha]\subseteq \g{g}_\alpha\ominus\g{v}_\alpha$.

Let $U$, $V\in\g{v}_\alpha$ and $T\in\g{s}_\g{t}$.
We choose $X\in{\g{a}\oplus\g{n}}$, such that $T+X\in\g{s}$. 
Since the action of $S$ is polar, we know that $[\g{s}_{\g{p}}^{\perp},\g{s}_{\g{p}}^{\perp}]$ is perpendicular to $\g{s}$. Thus,
\begin{align*}
0&{}=\langle[(1-\theta)U,(1-\theta)V],\,T+X\rangle
=\langle-(1+\theta)[\theta U,V],\,T\rangle\\
&{}=-2\langle [U,\theta V] ,\,T \rangle
=-2\langle U, [V,T]  \rangle.
\end{align*}
This proves $[\g{s}_{\g{t}},\g{v}_\alpha]\subseteq\g{g_{\alpha}\ominus\g{v}_\alpha}$.

Let $T\in\g{s}\cap\g{t}$ and $V\in\g{v}_\alpha$.
Since $\g{s}\cap\g{t}\subseteq\g{s}_\g{t}$, we have $[T,V]\in\g{g}_\alpha\ominus\g{v}_\alpha$.
Let $X\in\g{g}_\alpha\ominus\g{v}_\alpha$. 
Then there exists $T_X\in\g{t}$ such that $T_X+X\in\g{s}$.
Thus, $[T,X]=[T,\,T_X+X]\in\g{s}$, and hence, $\langle[T,X],V\rangle=0$.
We have proved $[\g{s}\cap\g{t},\g{v}_\alpha]=0$.

Let $T\colon\tilde{\g{s}}\to\g{s}_{\g{t}}\ominus(\g{s}\cap\g{t})$, $X\mapsto T_X$, be defined by $T_{X}+X\in\g{s}$.
This map is well-defined: if $T_X$, $T_X'\in\g{t}$ are such that $T_X+X$, $T_X'+X\in\g{s}$, subtracting,  $T_X-T_X'\in\g{s}\cap\g{t}$.
Note that $T$ is surjective.

Given a nonzero $V\in\g{v}_\alpha$, we define $\Phi_{V}\colon\g{g}_{\alpha}\ominus\g{v}_\alpha\to\g{g}_{\alpha}\ominus\g{v}_\alpha$ by $\Phi_{V}(X)=[T_{X},V]$. 
We prove that $\Phi_{V}$ is self-adjoint. 
Indeed, given $X$, $Y\in\g{g}_{\alpha}\ominus\g{v}_\alpha\subseteq\tilde{\g{s}}$, we obtain $[T_{X},Y]+[X,T_{Y}]+[X,Y]=[T_{X}+X,T_{Y}+Y]\in\g{s}$, which means 
\[
0=\langle V,[T_{X}+X,T_{Y}+Y]\rangle
=-\langle[T_{X},V],Y\rangle+\langle[T_{Y},V],X\rangle
=-\langle\Phi_V(X),Y\rangle+\langle X,\Phi_V(Y)\rangle.
\]

We now prove that $\Phi_{V}=0$. 
Assume this is not the case, so by the spectral theorem, there exists a nonzero vector $X\in\g{g}_{\alpha}\ominus\g{v}_\alpha$ and a nonzero constant $\lambda\in\mathbb{R}$ such that $\Phi_{V}(X)=\lambda X$. 

Observe that $[V,T_{[V,X]}]=0$. Indeed, $[V,T_{[V,X]}]\in\g{g}_{\alpha}\ominus\g{v}_\alpha$, and given any $Y\in\g{g}_{\alpha}\ominus\g{v}_\alpha$ we obtain 
\begin{align*}
0&{}=\langle [T_Y+Y,T_{[V,X]}+[V,X]],\,V\rangle
=\langle [T_{Y},[V,X]]+[Y,T_{[V,X]}],\,V \rangle\\
&{}=\langle [Y,T_{[V,X]}],\,V \rangle
=-\langle Y,\,[V,T_{[V,X]}]\rangle,
\end{align*}
which implies $[V,T_{[V,X]}]=0$. 

Now, consider $g=\Exp(\frac{1}{\lambda}V)$ and $Z=T_{X}+X-\frac{1}{2\lambda}\bigl(T_{[V,X]}+[V,X]\bigr)\in\g{s}$. Then
\begin{align*}
\Ad(g)Z
&{}=e^{\frac{1}{\lambda}\ad(V)}
\Bigl(T_{X}+X-\frac{1}{2\lambda}T_{[V,X]}-\frac{1}{2\lambda}[V,X]\Bigr) \\
&{}=T_{X}+X-\frac{1}{2\lambda}T_{[V,X]}-\frac{1}{2\lambda}[V,X] +\dfrac{1}{\lambda}
\bigl(-\lambda X + [V,X]\bigr)-\frac{1}{2\lambda^{2}}\lambda[V,X] \\
&{}=T_{X}-\dfrac{1}{2\lambda}T_{[V,X]}\in\Ad(g)(\g{s})\cap\g{k}.
\end{align*}

By Lemma~\ref{lemma:isotropyAlgebras}, we obtain $\Ad(g)Z\in\g{s}\cap\g{t}$, and thus, $Z\in \Ad(g^{-1})(\g{s}\cap\g{t})=\g{s}\cap\g{t}$, a contradiction. 
We conclude that $\Phi_{V}$ is the zero map for every $V\in\g{v}_\alpha$.
Since $\Phi_{\g{v}_\alpha}\equiv 0$ we have $[T_{X},V]=0$ for every $V \in \g{v}_{\alpha}$ and $X\in \g{g}_{\alpha}\ominus \g{v}_{\alpha}$.

Now, let $H\in\g{z}$ and $X\in\g{g}_{\alpha}\ominus\g{v}_{\alpha}$.
Since the vectors $T_{H}+H$ and $T_{X}+X$ are in $\g{s}$, their Lie bracket $[T_{H}+H,T_{X}+X]=[T_{H},X]+\alpha(H)X$ is also in $\g{s}$, and because $X\in\g{g}_{\alpha}\ominus \g{v}_{\alpha}$ we deduce that $[T_{H},X]\in\g{g}_{\alpha}\ominus \g{v}_{\alpha}$.
As a consequence, $\ad(T_{H})$ preserves $\g{g}_{\alpha}\ominus \g{v}_{\alpha}$, so it also preserves $\g{v}_{\alpha}$ as it is skew-symmetric.
Similarly, suppose that $Y\in\g{g}_{\lambda}$ is any vector, with $\lambda\in\Sigma^{+}\setminus \{\alpha\}$.
We see that $[T_{Y},X]+[Y,T_{X}]+[Y,X]=[T_{Y}+Y,T_{X}+X]\in\g{s}$, so taking the inner product with any $V\in\g{g}_{\alpha}$ we obtain that $0=\langle [T_{Y},X],V\rangle$.
This means that $\ad(T_{Y})(\g{g}_{\alpha}\ominus\g{v}_{\alpha})$ is contained in $\g{g}_{\alpha}\ominus \g{v}_{\alpha}$, and the skew-symmetry of $\ad(T_{Y})$ yields $\ad(T_{Y})(\g{v}_{\alpha})\subseteq \g{v}_{\alpha}$.

All in all, we have seen that $[\g{s}_\g{t}\ominus(\g{s}\cap\g{t}),\,\g{v}_\alpha]\subseteq \g{v}_{\alpha}$.
Let $T\in\g{s}_{\g{t}}\ominus (\g{s}\cap \g{t})$ and $V\in\g{v}_{\alpha}$.
Then, the skew-symmetry of $\ad(T)$ implies that $[T,V]\in \g{v}_{\alpha}\ominus \R V$.
Choose $X\in\g{s}_{\g{a}\oplus\g{n}}$ such that $T+X\in\g{s}$, and let $W\in\g{v}_{\alpha}\ominus \R V$ be arbitrary.
As the action of $S$ is polar and $\g{v}_{\alpha}$ is abelian, we see that
\[
	-2[\theta V,W]=[(1-\theta)V,(1-\theta)W]\in\g{k}_{0}
\]
is orthogonal to $\g{s}$, and as a consequence we deduce that
\[
	0=\langle T+X,[\theta V,W] \rangle =\langle [T,V],W \rangle,
\]
which means that $[T,V]=0$, so $[\g{s}_\g{t}\ominus(\g{s}\cap\g{t}),\g{v}_\alpha]=0$.

Combining what we saw in the previous paragraph with $[\g{s}\cap\g{t},\g{v}_\alpha]=0$, we arrive at $[\g{s}_\g{t},\g{v}_\alpha]=0$.
Therefore, by skew-symmetry of the elements of $\g{t}$, we get 
$\langle[\g{s}_\g{t},\g{g}_\alpha\ominus\g{v}_\alpha],\g{v}_\alpha\rangle=\langle \g{g}_\alpha\ominus\g{v}_\alpha,[\g{s}_\g{t},\g{v}_\alpha]\rangle=0$.
Since $\g{t}$ normalizes $\g{g}_\alpha$, we finally get $[\g{s}_\g{t},\g{g}_\alpha\ominus\g{v}_\alpha]\subseteq \g{g}_\alpha\ominus\g{v}_\alpha$, which in turn implies that $\hat{\g{s}}=\g{s}_{\g{t}}\oplus\g{z}\oplus(\g{n}\ominus\g{v}_\alpha)=\g{s}_\g{t}\oplus\g{s}_{\g{a}\oplus\g{n}}$ is a Lie subalgebra of $\g{t}\oplus\g{a}\oplus\g{n}$.

We can therefore consider the connected subgroup $\hat{S}$ of $G$ whose Lie algebra is $\hat{\g{s}}$. 
We prove that $\hat{S}$, $S$ and $\tilde{S}$ have the same orbits. 
Note that a priori we do not know if $\hat{S}$ is closed, and thus the action of $\hat{S}$ may not be proper.
Since $S\subseteq \hat{S}$ and $\g{s}_{\g{a}\oplus\g{n}}=\hat{\g{s}}_{\g{a}\oplus\g{n}}=\g{z}\oplus(\g{n}\ominus\g{v}_\alpha)$, we deduce that $S\cdot o = \hat{S}\cdot o$. 
The same argument may be applied to see that $\hat{S}\cdot o =\tilde{S}\cdot o$. 
In particular, $S\cdot o = \hat{S}\cdot o = \tilde{S}\cdot o$ is simply connected (because $AN$ is an exponential Lie group acting simply transitively on $M$), which means that the isotropy subgroups $S\cap K$ and $\hat{S}\cap K$ are connected. 
As a consequence, the slice representation of $\hat{S}$ at $o$ is trivial because $[\hat{\g{s}}\cap\g{k},\g{s}_{\g{p}}^{\perp}]
=[\g{s}_{\g{t}},(\g{a}\ominus\g{z})\oplus(1-\theta)\g{v}_\alpha]=(1-\theta)[\g{s}_\g{t},\g{v}_\alpha]=0$. 
Hence, using Lemma~\ref{lemma:OrbitEquivalenceNonClosed} twice, the groups $S$, $\hat{S}$, and $\tilde{S}$ act with the same orbits .
\end{proof}

Since we can assume that $\g{s}_{\g{p}}^{\perp}\cap \g{p}^{1}$ has dimension one or two, we tackle these two possibilities separately.

\subsection{The case $\g{s}_{\g{p}}^{\perp}\subseteq \g{p}^{1}$}\hfill
	
Assume $\g{s}_{\g{p}}^{\perp}$ is contained in $\g{p}^{1}$. 
We have that $\g{a}$ and $\g{n}\ominus\g{n}^{1}$ are subspaces of $\g{s}_{\g{a}\oplus\g{n}}$. 
A direct application of Lemma~\ref{lemma:projectionImpliesContained} gives $\g{n}\ominus\g{n}^{1}\subseteq \g{s}$.
Let $\xi=\sum_{\alpha\in\Lambda}(1-\theta)\xi_{\alpha}$ and $\eta=\sum_{\alpha\in\Lambda}(1-\theta)\eta_{\alpha}$ be orthonormal vectors in $\g{s}_{\g{p}}^{\perp}$, where $\xi_{\alpha}$, $\eta_{\alpha}\in\g{g}_{\alpha}$. 
Since the action is polar nonhyperpolar, $[\xi,\eta]=(1+\theta)\bigl(\sum_{\alpha,\beta\in\Lambda}
[\xi_{\alpha},\eta_{\beta}]-[\theta\xi_{\alpha},\eta_{\beta}]\bigr)$ is a nonzero vector orthogonal to $\g{s}$. 
By using the fact that $\g{n}\ominus\g{n}^{1}\subseteq \g{s}$ and $[\theta\xi_{\alpha},\eta_{\beta}]=0$ when $\beta\neq \alpha$, we deduce $[\xi_{\alpha},\eta_{\alpha}]=0$, and $[\xi_{\alpha},\eta_{\beta}]+[\xi_{\beta},\eta_{\alpha}]=0$.
Thus, $[\xi,\eta]=-(1+\theta)\sum_{\alpha\in\Lambda}
[\theta\xi_{\alpha},\eta_{\alpha}]$.

Since $\g{s}_\g{p}^\perp$ is a two-dimensional Lie triple system, it determines a totally geodesic submanifold that is isometric to a real hyperbolic space.
Hence, there exists $C>0$ such that $\ad(\xi)^2\eta=C\eta$.
Thus, we have for every $\alpha\in\Lambda$,
\begin{align*}
C\langle\xi_\alpha,\eta_\alpha\rangle
&{}=C\langle\xi_\alpha,\eta\rangle=\langle\xi_\alpha,\ad(\xi)^2\eta\rangle
=\langle[\xi,\xi_\alpha],[\xi,\eta]\rangle\\[1ex]
&{}=-\Bigl\langle \sum_{\beta\in\Lambda}[\xi_\beta,\xi_\alpha]-[\theta\xi_\alpha,\xi_\alpha],(1+\theta)\sum_{\gamma\in\Lambda}
[\theta\xi_{\gamma},\eta_{\gamma}]\Bigr\rangle\\
&{}=\Bigl\langle \lvert\xi_{\alpha}\rvert^{2}H_{\alpha},(1+\theta)\sum_{\beta\in\Lambda}[\theta\xi_{\beta},\eta_{\beta}] \Bigr\rangle=0.
\end{align*}

\begin{proposition}\label{prop:normalVectorsInOneGalpha}
$\g{s}_{\g{a}\oplus\g{n}}=(\g{a}\oplus\g{n})\ominus\g{v}_\alpha$, where $\g{v}_\alpha$ is an abelian subspace of $\g{g}_{\alpha}$
\end{proposition}

\begin{proof}
Let $\lambda$, $\mu\in\Lambda$ with $\xi_{\lambda}$, $\xi_{\mu}\neq 0$. 
Since $\langle\xi_{\alpha},\eta_{\alpha}\rangle=0$ for all simple roots $\alpha\in \Lambda$, taking $x=-\lvert\xi_{\lambda}\rvert^{2}/\lvert\xi_{\mu}\rvert^{2}<0$ we have $\xi_{\lambda}+x\xi_{\mu}\in\g{s}_{\g{a}\oplus\g{n}}$. 
Hence, we may find a $T\in\g{t}$ such that $T+\xi_{\lambda}+x\xi_{\mu}\in\g{s}$. 
On the other hand, choose a vector $H\in \g{a}$ with $\lambda(H)=0$ and $\mu(H)\neq 0$. 
Since $\g{a}\subseteq \g{s}_{\g{a}\oplus\g{n}}$, there exists $T'\in \g{t}$ such that $T'+H\in \g{s}$. 
As a consequence, $[T',\xi_{\lambda}]+x[T',\xi_{\mu}]+x\mu(H)\xi_{\mu}
=[T'+H,T+\xi_{\lambda}+x\xi_{\mu}]\in\g{s}$. 
This means that $0=\langle[T',\xi_{\lambda}]+x[T',\xi_{\mu}]+x\mu(H)\xi_{\mu},
\sum_{\alpha}(1-\theta)\xi_{\alpha} \rangle=x\mu(H)\lvert\xi_{\mu}\rvert^{2}$, which is a contradiction. 
Thus, $\xi=(1-\theta)\xi_{\alpha}$ for a fixed simple root $\alpha\in \Lambda$. 
The same argument can be applied to conclude that $\eta=(1-\theta)\xi_{\beta}$ for a simple root $\beta\in\Lambda$. 
We now prove that $\alpha=\beta$. 
Indeed, if $\alpha\neq \beta$, we have $[\xi,\eta]=(1-\theta)([\xi_{\alpha},\eta_{\beta}]
-[\theta\xi_{\alpha},\eta_{\beta}])=0$, contradicting the fact that $S$ has a non flat section.

Finally, $\g{v}_\alpha=\R\xi_\alpha\oplus\R\eta_\alpha$ is abelian due to the discussion at the beginning of this case.
\end{proof}

The expression obtained in Proposition~\ref{prop:caseDim1OrbitEquivalence} together with  Proposition~\ref{prop:orbit-equivalence} with $\g{z}=\g{a}$ imply now that the action of $S$ is orbit equivalent to item~(\ref{th:main:a+n-v}) of Theorem~\ref{th:MainTheoremList}.

\subsection{The case $\dim (\g{s}_{\g{p}}^{\perp}\cap\g{p}^{1})=1$.}\hfill
	
In this setting, we can choose two orthonormal vectors $\xi=\xi_{0}+\sum_{\alpha\in\Lambda}(1-\theta)\xi_{\alpha}$ and $\eta=\sum_{\alpha\in\Lambda}(1-\theta)\eta_{\alpha}$ that span $\g{s}_{\g{p}}^{\perp}$, where $\xi_{0}\in\g{a}$ is nonzero and $\xi_{\alpha}$, $\eta_{\alpha}\in\g{g}_{\alpha}$ for each $\alpha\in\Lambda$.

Since $\alpha-\beta\notin\Sigma$ for $\alpha$, $\beta\in\Lambda$, $\alpha\neq\beta$, we have
\[
[\xi,\eta]
=(1+\theta)\Bigl(\sum_{\alpha\in\Lambda}\bigl(\alpha(\xi_0)\eta_\alpha-[\theta\xi_\alpha,\eta_\alpha]\bigr)
+\sum_{\alpha,\beta\in\Lambda}[\xi_\alpha,\eta_\beta]\Bigr).
\]
Recall that $\ad(\xi)^{2}\eta=C\eta$ for a positive constant $C\in \mathbb{R}$. In particular, for any $H\in\g{a}$, 
\begin{align*}
0
&{}=\langle C\eta,\,H\rangle
=\langle\ad(\xi)^2\eta,\,H\rangle
=\langle[\xi,\eta],\,[\xi,H]\rangle=-\Bigl\langle [\xi,\eta],\,
(1+\theta)\sum_{\gamma\in\Lambda}\gamma(H)\xi_\gamma\Bigr\rangle\\
&{}=-2\sum_{\alpha\in\Lambda}\alpha(\xi_0)\alpha(H)\langle\xi_\alpha,\eta_\alpha\rangle
=\Bigl\langle -2\sum_{\alpha\in\Lambda}\alpha(\xi_0)\langle\xi_\alpha,\eta_\alpha\rangle H_\alpha,\,H\Bigr\rangle,
\end{align*}
which implies $-2\sum_{\alpha\in\Lambda}\alpha(\xi_{0})\langle\xi_{\alpha},\eta_{\alpha}\rangle H_{\alpha}=0$. Therefore, 
\begin{equation}\label{eq:inner-product-alpha}
\alpha(\xi_{0})\langle \xi_{\alpha},\eta_{\alpha} \rangle = 0, \text{ for every $\alpha\in\Lambda$.}
\end{equation}
	
On the other hand, since $\ad(\eta)^{2}\xi=C\xi$, for any $H\in\g{a}$,
\begin{align*}
C\langle\xi_0,H\rangle
&{}=\langle C\xi,H\rangle
=\langle \ad(\eta)^{2}\xi,\,H \rangle
=\langle [\eta,\xi],\, [\eta,H]]\rangle
=\Bigl\langle [\xi,\eta],\,
(1+\theta)\sum_{\gamma\in\Lambda}\gamma(H)\eta_\gamma\Bigr\rangle \\
&{}=2 \sum_{\alpha\in\Lambda}\alpha(\xi_0)\alpha(H)\lvert\eta_\alpha\rvert^2
=\Bigl\langle 2\sum_{\alpha\in\Lambda}\alpha(\xi_0)\lvert\eta_\alpha\rvert^2 H_\alpha,\,H\Bigr\rangle,
\end{align*}
we obtain
\begin{equation}\label{eq:a-component}
C\xi_0=2\sum_{\alpha\in\Lambda}\alpha(\xi_{0})\lvert\eta_{\alpha}\rvert^{2} H_{\alpha}. 
\end{equation}
	
Similarly, let $\alpha,\beta\in\Lambda$ be arbitrary, and $X\in\g{g}_{\alpha+\beta}$. 
Then, 	
\begin{align*}
0&{}=\langle C\xi,X\rangle
=\langle\ad(\eta)^2\xi,\,X\rangle
=\langle[\eta,\xi],\,[\eta,X]\rangle
=-\Bigl\langle [\xi,\eta],\,
\sum_{\mu\in\Lambda}\bigl([X,\eta_\mu]-[X,\theta\eta_\mu]\bigr)\Bigr\rangle\\
&{}=-\Bigl\langle \sum_{\gamma\in\Lambda}\gamma(\xi_0)\eta_\gamma,\,\sum_{\delta\in\Lambda}[X,\theta\eta_\delta]\Bigr\rangle
=\sum_{\gamma,\delta\in\Lambda}\gamma(\xi_0)\langle [\eta_\gamma,\eta_\delta],\,X\rangle\\[1ex]
&{}=\langle \alpha(\xi_{0})[\eta_{\alpha},\eta_{\beta}]+\beta(\xi_{0})[\eta_{\beta},\eta_{\alpha}],X \rangle.
\end{align*}
Consequently, for any two simple roots $\alpha,\beta\in\Lambda$, $[(\alpha-\beta)(\xi_{0})\,\eta_{\alpha},\eta_{\beta}]=0$.
	
\begin{lemma}\label{lemma:components-orthogonal}
We have $\langle \xi_{\alpha},\eta_{\alpha} \rangle=0$ for all $\alpha\in\Lambda$.
\end{lemma}

\begin{proof}
We define $\Psi=\{ \alpha\in\Lambda \colon \eta_{\alpha}\neq 0 \}$. 
We show that $\langle\xi_{\alpha},\eta_{\alpha}\rangle=0$ for each $\alpha\in\Psi$.

From~\eqref{eq:a-component}
the map $\alpha\in\Psi\mapsto\alpha(\xi_{0})$ cannot be identically zero.
Thus, fix $\alpha\in\Psi$ such that $\alpha(\xi_0)\neq 0$. 
Hence~\eqref{eq:inner-product-alpha} already implies $\langle\xi_\alpha,\eta_\alpha\rangle=0$.

Assume $\beta\in\Psi$ satisfies $\langle\xi_\beta,\eta_\beta\rangle\neq 0$.
In particular, from \eqref{eq:inner-product-alpha} we have $\beta(\xi_{0})=0$.
If $\langle\alpha,\beta\rangle\neq 0$, the linear map $\ad(\eta_{\beta})\colon\g{g}_{\alpha}\to\g{g}_{\alpha+\beta}$ is injective. 
From $[(\alpha-\beta)(\xi_{0})\,\eta_{\alpha},\eta_{\beta}]=0$ we deduce  $(\alpha-\beta)(\xi_{0})\eta_{\alpha}=0$, so $\alpha(\xi_{0})=\beta(\xi_{0})=0$, contradiction. 
Thus, $\langle\alpha,\beta\rangle=0$.
In particular, $H_{\beta}\in\g{s}_{\g{a}\oplus\g{n}}$, so there exists $T\in\g{t}$ such that $T+H_{\beta}\in\g{s}$. 
On the other hand, $\xi_{0}+x\eta_{\alpha}+y\eta_{\beta}\in\g{s}_{\g{a}\oplus\g{n}}$ for $y=-\lvert\xi_{0}\rvert^{2}/\langle \xi_{\beta},\eta_{\beta}\rangle \neq 0$ and $x=-y\lvert\eta_{\beta}\rvert^{2}/\lvert\eta_{\alpha}\rvert^{2}\neq 0$, so $T'+\xi_{0}+x\eta_{\alpha}+y\eta_{\beta}\in\g{s}$ for an adequate $T'\in\g{t}$. 
Thus,
\begin{align*}
0&{}=\langle[T+H_{\beta},T'+\xi_{0}+x\eta_{\alpha}+y\eta_{\beta}],\,\eta\rangle\\
&{}=\langle x[T,\eta_{\alpha}]+y[T,\eta_{\beta}]+y\lvert\beta\rvert^{2}\eta_{\beta},\,\eta\rangle
=y\lvert\beta\rvert^2\lvert\eta_\beta\rvert^{2},
\end{align*}
which gives us a contradiction. 
Therefore, $\langle\xi_\beta,\eta_\beta\rangle=0$ for all $\beta\in\Lambda$.
\end{proof}

\begin{proposition}\label{prop:caseDim2NormalVectors}
There exists a simple root $\alpha\in\Lambda$ and a constant $a\in\mathbb{R}$ such that $\xi=aH_{\alpha}+(1-\theta)\xi_{\alpha}$. If $\xi_{\alpha}=0$ (that is, if $\xi\in\g{a}$), then $\eta=(1-\theta)\eta_{\alpha}$.
\end{proposition}

\begin{proof}
Firstly, suppose $\xi\in\g{a}$. Then, a direct application of Lemma~\ref{lemma:normalap1} implies that $\eta=(1-\theta)\eta_{\alpha}$ for a simple root $\alpha\in\Lambda$. In particular,~\eqref{eq:a-component} is reduced to $C\xi_{0}=2\alpha(\xi_{0})\lvert\eta_{\alpha}\rvert^{2}H_{\alpha}$, so $\xi_{0}\in\mathbb{R}H_{\alpha}$, and the proposition follows.

Now, assume $\xi_{0}\notin\g{a}$, and let $\alpha\in\Lambda$ be a simple root such that $\xi_{\alpha}\neq 0$. 
We prove that $\xi=aH_{\alpha}+(1-\theta)\xi_{\alpha}$.

Suppose that $\xi_{0}$ is not proportional to $H_{\alpha}$. 
Then there exists $H\in\g{a}$ such that $\langle H,\xi_{0} \rangle=0$ and $\alpha(H)\neq0$. 
As a consequence, $H\in\g{s}_{\g{a}\oplus\g{n}}$, and there exists $T\in\g{t}$ for which $T+H\in\g{s}$. 
On the other hand, $\xi_{0}+x\xi_{\alpha}\in\g{s}_{\g{a}\oplus\g{n}}$ for $x=-\lvert\xi_{0}\rvert^{2}/\lvert\xi_{\alpha}\rvert^{2}<0$ (because $\langle\xi_{\alpha},\eta_{\alpha}\rangle=0$ by Lemma~\ref{lemma:components-orthogonal}), and we may choose $T'\in\g{t}$ such that $T'+\xi_{0}+x\xi_{\alpha}\in\g{s}$. Thus, $[T+H,T'+\xi_{0}+x\xi_{\alpha}]=x[T,\xi_{\alpha}]+x\alpha(H)\xi_{\alpha}\in\g{s}$. Taking the inner product with $\xi$ yields $x\alpha(H)\lvert\xi_{\alpha}\rvert^{2}=0$, a contradiction.

Hence $\xi_{0}\in\mathbb{R}H_{\alpha}$ for any $\alpha\in\Lambda$.
The fact that simple roots are linearly independent together with \eqref{eq:a-component} implies $\xi_{\beta}=0$ for every $\beta\in\Lambda\setminus\{\alpha\}$. 
\end{proof}

So far we have proved that $\xi$ must take the form $aH_{\alpha}+(1-\theta)\xi_{\alpha}$ for a nonzero $a\in\R$ and $\xi_{\alpha}\in\g{g}_{\alpha}$ (which may be zero). If $\xi_{\alpha}=0$, then we also know that $\eta=(1-\theta)\eta_{\alpha}$. If $\xi_{\alpha}\neq 0$, then the third statement of Lemma~\ref{lemma:normalap1} implies that the action of $S$ is orbit equivalent to an action of another closed connected subgroup $\tilde{S}$ for which the normal space of $\tilde{S}\cdot o$ at $o$ takes the form $\tilde{\g{s}}_{\g{p}}^{\perp}=\{ \xi_{\alpha},bH_{\alpha}+(1-\theta)\nu_{\alpha} \}$ for a constant $b\in \mathbb{R}$ and a $\nu_{\alpha}\in\g{g}_{\alpha}$. 
Because of this, we may assume without loss of generality that $\g{s}_\g{p}^\perp$ is spanned by two orthogonal vectors $\xi=a H_\alpha+(1-\theta)\xi_\alpha$, $\eta=(1-\theta)\eta_{\alpha}$, with $a\neq 0$, $\xi_\alpha$, $\eta_\alpha\in\g{g}_\alpha$.
Recall from Lemma~\ref{lemma:normalap1} that $[\xi_\alpha,\eta_\alpha]=0$.

The key to finishing the proof lies in the following result:

\begin{proposition}\label{prop:caseDim1OrbitEquivalence}
Assume $\g{s}_{\g{p}}^{\perp}=\operatorname{span} \{ \xi,\eta \}$, where $\xi=a H_{\alpha}+(1-\theta)\xi_{\alpha}$ and $\eta=(1-\theta)\eta_{\alpha}$, where $a\neq 0$, $\xi_{\alpha}$, $\eta_{\alpha}\in\g{g}_{\alpha}$ are orthogonal commuting vectors, and $\alpha\in\Lambda$. 
Then:
\begin{enumerate}[{\rm (i)}]
\item If $\xi_{\alpha}=0$, then the action of $S$ has the same orbits as the action of the connected subgroup of $G$ whose Lie algebra is $(\g{a}\ominus\mathbb{R}H_{\alpha})\oplus(\g{n}\ominus\mathbb{R}\eta_{\alpha})$.\label{prop:caseDim1OrbitEquivalence:ker+n-l}
\item If $\xi_{\alpha}\neq 0$, then there exists an abelian subspace $\g{v}_\alpha\subseteq\g{g}_{\alpha}$ such that the action of $S$ is orbit equivalent to the action of the connected subgroup of $G$ whose Lie algebra is $\g{a}\oplus(\g{n}\ominus\g{v}_\alpha)$.\label{prop:caseDim1OrbitEquivalence:a+n-v}
\end{enumerate}
\end{proposition}

\begin{proof}
If $\xi_\alpha=0$, then $\g{s}_\g{p}^\perp=\R H_\alpha\oplus\R(1-\theta)\eta_\alpha$ and $\g{s}_{\g{a}\oplus\g{n}}=(\g{a}\ominus\R H_\alpha)\oplus(\g{n}\ominus\R\eta_\alpha)$.
Then, statement~(\ref{prop:caseDim1OrbitEquivalence:ker+n-l}) follows directly from Proposition~\ref{prop:orbit-equivalence}.

We prove~(\ref{prop:caseDim1OrbitEquivalence:a+n-v}). 
We consider the element $g=\Exp(-\frac{a}{\lvert\xi_{\alpha}\rvert^{2}}\xi_{\alpha})\in N$. 
Since $\g{s}$ is orthogonal to $aH_{\alpha}+\xi_{\alpha}$ and $\eta_{\alpha}$, it follows that $\Ad(g)\g{s}$ is orthogonal to $\Ad(g^{-1})^{*}(aH_{\alpha}+\xi_{\alpha})$ and $\Ad(g^{-1})^{*}\eta_{\alpha}$. 
By direct computation,
\begin{align*}
\Ad(g^{-1})^{*}(aH_{\alpha}+\xi_{\alpha})
&{}=e^{-\frac{a}{\lvert\xi_{\alpha}\rvert^{2}}\ad(\theta\xi_{\alpha})}(aH_{\alpha}+\xi_{\alpha})
\equiv \xi_{\alpha} \pmod{ \theta\g{n}},\\
\Ad(g^{-1})^{*}\eta_{\alpha}
&{}=e^{-\frac{a}{\lvert\xi_{\alpha}\rvert^{2}}\ad(\theta\xi_{\alpha})}\eta_{\alpha}
\equiv \eta_{\alpha}-\frac{a}{\lvert\xi_{\alpha}\rvert^{2}}[\theta\xi_{\alpha},\eta_{\alpha}]\pmod{\theta\g{n}},
\end{align*}
and since $\Ad(g)\g{s}\subseteq\g{t}\oplus\g{a}\oplus\g{n}$, it follows that the vectors $\xi_{\alpha}$ and $\eta_{\alpha}-\frac{a}{\lvert\xi_{\alpha}\rvert^{2}}[\theta\xi_{\alpha},\eta_{\alpha}]$ are orthogonal to $\Ad(g)\g{s}$. 
On the other hand, the action of $S$ is polar, so $[\xi,\eta]=a\lvert\alpha\rvert^{2}(1+\theta)\eta_{\alpha}-2[\theta\xi_{\alpha},\eta_{\alpha}]$ is also orthogonal to $\g{s}$. 
As a consequence, $\langle a\lvert\alpha\rvert^{2}\eta_{\alpha}-2[\theta\xi_{\alpha},\eta_{\alpha}],\g{s}\rangle=0$. We deduce that
\[
\Ad(g^{-1})^{*}(a\lvert\alpha\rvert^{2}\eta_{\alpha}-2[\theta\xi_{\alpha},\eta_{\alpha}])
\equiv a\lvert\alpha\rvert^{2}\eta_{\alpha}
-\Bigl(2+\frac{a^2\lvert\alpha\rvert^{2}}{\lvert\xi_{\alpha}\rvert^{2}}\Bigr)[\theta\xi_{\alpha},\eta_{\alpha}]
\pmod{\theta\g{n}}
\]
is also orthogonal to $\Ad(g)\g{s}$. 
Because $\theta\g{n}$ is already orthogonal to $\g{s}$,
it follows that $a\lvert\alpha\rvert^{2}\eta_{\alpha}-\left(2+\frac{a^2\lvert\alpha\rvert^{2}}{\lvert\xi_{\alpha}\rvert^{2}}\right)[\theta\xi_{\alpha},\eta_{\alpha}]$ is perpendicular to $\g{s}$.
Since
\[
\begin{vmatrix}
1 & -\frac{a}{\lvert\xi_{\alpha}\rvert^{2}} \\
a\lvert\alpha\rvert^{2} & -2-\frac{a^2\lvert\alpha\rvert^{2}}{\lvert\xi_{\alpha}\rvert^{2}}
\end{vmatrix}
=-2<0,
\]
we deduce that $\eta_{\alpha}$ and $[\theta\xi_{\alpha},\eta_{\alpha}]$ are both orthogonal to $\g{s}$. 
We conclude that $(\Ad(g)\g{s})_{\g{p}}^{\perp}=(1-\theta)\g{v}_\alpha$, where $\g{v}_\alpha=\Span\{ \xi_{\alpha},\eta_{\alpha}\}$, and thus, $gSg^{-1}$ and the connected subgroup of $G$ whose Lie algebra is $\g{a}\oplus(\g{n}\ominus\g{v}_\alpha)$ act with the same orbits due to Proposition~\ref{prop:orbit-equivalence}.
\end{proof}

The proof of Theorem~\ref{th:MainTheoremList} follows now from the observation that Proposition~\ref{prop:caseDim1OrbitEquivalence}(\ref{prop:caseDim1OrbitEquivalence:ker+n-l}) corresponds to case~(\ref{th:main:ker+n-l}), and Proposition~\ref{prop:caseDim1OrbitEquivalence}(\ref{prop:caseDim1OrbitEquivalence:a+n-v}) corresponds to case~(\ref{th:main:a+n-v}).

\end{document}